\documentclass{amsart}
\usepackage{latexsym}
\usepackage{amsmath}
\usepackage{amsthm}
\usepackage{amssymb}
\usepackage{hyperref}

\newcommand{\ds}{\displaystyle}
\newcommand{\vip}{\vskip0.15cm}
\newcommand{\ala}{\nonumber \\}
\newcommand{\indiq}{1\!\! 1}
\newcommand{\e}{{\varepsilon}}

\newcommand{\cH}{{{\mathcal H}}}
\newcommand{\cA}{{{\mathcal A}}}
\newcommand{\ctA}{{\tilde {\mathcal A}}}
\newcommand{\cF}{{{\mathcal F}}}
\newcommand{\cJ}{{{\mathcal J}}}
\newcommand{\cK}{{{\mathcal K}}}
\newcommand{\cP}{{{\mathcal P}}}

\newcommand{\dt}{{\mathbb{D}_T}}
\newcommand{\sm}{{s-}}

\newcommand{\rr}{{\mathbb{R}}}
\newcommand{\Sp}{{\mathbb{S}}}
\newcommand{\rd}{{\mathbb{R}^3}}

\newcommand{\pdeux}{{{\mathcal P}_2(\rd)}}
\newcommand{\lipdeux}{L^\infty([0,T],\pdeux)}
\newcommand{\lunlp}{L^1([0,T], L^p(\rd))}
\newcommand{\lunjgamma}{L^1([0,T], \cJ_\gamma)}

\newcommand{\intot}{\ds\int_0^t}
\newcommand{\intrd}{\ds\int_{\rd} \!\!\!}
\newcommand{\intzi}{\ds\int_0^\infty}
\newcommand{\intrdd}{\ds\int_{\rd\times\rd}\!\!\!\!\!\!\!\!\!}
\newcommand{\intzp}{\ds\int_0^\pi}
\newcommand{\intzdp}{\ds\int_0^{2\pi}}

\newcommand{\tf}{{\tilde f}}
\newcommand{\tg}{{\tilde g}}
\newcommand{\tv}{{\tilde v}}
\newcommand{\tc}{{\tilde c}}
\newcommand{\tN}{{\tilde N}}
\newcommand{\tV}{{\tilde V}}

\newcommand{\tB}{{\tilde B}}
\newcommand{\tI}{{\tilde I}}
\newcommand{\tZ}{{\tilde Z}}
\newcommand{\talpha}{{\tilde \alpha}}

\renewcommand{\theequation}{\thesection.\arabic{equation}}

\newtheorem{theo}{\indent Theorem}[section]
\newtheorem{prop}[theo]{\indent Proposition}

\newtheorem{lem}[theo]{\indent Lemma}
\newtheorem{defin}[theo]{\indent Definition}
\newtheorem{cor}[theo]{\indent Corollary}
\newtheorem{nota}[theo]{\indent Notation}

\begin{document}

\title[Very singular Boltzmann equations]
{On the uniqueness for the spatially homogeneous 
Boltzmann equation with a strong angular singularity}

\author{Nicolas Fournier$^1$, H\'el\`ene Gu\'erin$^2$}

\footnotetext[1]{Centre de Math\'ematiques,
Facult\'e de Sciences et Technologies,
Universit\'e Paris~12, 61, avenue du G\'en\'eral de Gaulle, 94010 Cr\'eteil 
Cedex, France, {\tt nicolas.fournier@univ-paris12.fr}}

\footnotetext[2]{IRMAR,
Univ. Rennes 1, Campus de Beaulieu, 35042 Rennes Cedex, France,
\texttt{helene.guerin@univ-rennes1.fr}}

\def\shortauthorname{Nicolas Fournier, H\'el\`ene Gu\'erin}

\def\abstractname{Abstract}

\begin{abstract}
We prove an inequality on the Wasserstein distance with quadratic cost
between two solutions of the spatially homogeneous Boltzmann
equation without angular cutoff, from which we deduce 
some uniqueness results.
In particular, we obtain a local (in time) 
well-posedness result in the case of (possibly very) soft potentials. 
A global well-posedeness result is shown for all
regularized hard and soft potentials without angular cutoff.
Our uniqueness result seems to be the first one
applying to a strong angular singularity, except in the
special case of Maxwell molecules.

\noindent Our proof relies on the ideas of Tanaka \cite{tanaka}:
we give a probabilistic interpretation of the Boltzmann equation
in terms of a stochastic process. 
Then we show how to couple two such processes started with two different
initial conditions, in such a way that they almost surely remain close
to each other.
\end{abstract}

\maketitle

\smallskip

\smallskip

%\tableofcontents

%%%%%%%%%%%%%%%%%%%%%%%%%%%%%%%%%%%%%%%%%%%%%%%%%%%%%%%%%%%%%%%%%
%%%%
%%%%            INTRODUCTION
%%%%
%%%%%%%%%%%%%%%%%%%%%%%%%%%%%%%%%%%%%%%%%%%%%%%%%%%%%%%%%%%%%%%%%

\section{Introduction and main results} 
\setcounter{equation}{0}

\subsection{The Boltzmann equation}  
Let $f(t,v)$ be the density of particles with velocity $v\in \rd$ at time 
$t\geq 0$ in a spatially homogeneous dilute gas. Then under some
assumptions, $f$ solves 
the Boltzmann equation
\begin{eqnarray} \label{be}
\partial_t f_t(v) = \intrd dv_* \int_{\Sp^{2}} d\sigma B(|v-v_*|,\theta)
\big[f_t(v')f_t(v'_*) -f_t(v)f_t(v_*)\big],
\end{eqnarray}
where the pre-collisional velocities are given by
\begin{equation}\label{vprimeetc}
v'=\frac{v+v_*}{2} + \frac{|v-v_*|}{2}\sigma, \quad 
v'_*=\frac{v+v_*}{2} -\frac{|v-v_*|}{2}\sigma
\end{equation}
and $\theta$ is the so-called {\em deviation angle} defined by 
$\cos \theta = \frac{(v-v_*)}{|v-v_*|} \cdot \sigma$.
The {\em collision kernel} $B=B(|v-v_*|,\theta)=B(|v'-v'_*|,\theta)$ 
depends on the nature of the interactions between particles. 
\vip
This equation is quite natural:
it says that for each $v\in \rd$, new particles with velocity
$v$ appear due to a collision between two particles with velocities
$v'$ and $v'_*$, at rate $B(|v'-v'_*|,\theta)$, while 
particles with velocity $v$ disappear because they collide
with another particle with velocity $v_*$, at rate
$B(|v-v_*|,\theta)$.
See Desvillettes \cite{desvillettes} and Villani 
\cite{villani} for more much more details. 

\vip

Since the collisions are assumed to be elastic, 
conservation of mass, 
momentum and kinetic energy
hold at least formally for solutions to (\ref{be}), that is for all $t\geq 0$,
\begin{equation}
\intrd  f_t(v) \, \phi(v) \, dv 
= \intrd f_0(v) \, \phi(v) \, dv, \qquad \phi(v) = 1, v, |v|^2 ,
\end{equation}
and we classically may assume without loss of generality that 
$\int_{\rd} f_0(v) \, dv=1$.

\subsection{Assumptions on the collision kernel} 
We will assume that for some  functions 
$\Phi: \rr_+ \mapsto \rr_+$ and $\beta:(0,\pi] \mapsto (0,\infty)$,
\renewcommand\theequation{{\bf A1}}
\begin{equation}
B(|v-v_*|,\theta)\, \sin \theta =\Phi(|v-v_*|) \, \beta(\theta).
\end{equation}
\renewcommand\theequation{\thesection.\arabic{equation}}

In the case of an interaction potential $V(s)=1/r^s$, 
with $s\in (2,\infty)$, one has 
\begin{equation} \label{betapuissance}
\Phi(z) =   z^{\gamma}, 
\quad \beta(\theta) \stackrel{0}{\sim} \mbox{cst} \, \theta^{-1-\nu}, \; \; 
\mbox{ with } 
\gamma=\frac{s-5}{s-1}\in (-3,1), \quad \nu=\frac{2}{s-1}\in(0,2).
\end{equation}
On classically names {\em hard potentials} the case when 
$\gamma\in (0,1)$ ({\em i.e.}, $s>5$), 
{\em Maxwellian molecules} the case when $\gamma=0$ ({\em i.e.}, $s=5$), 
{\em moderately soft potentials} the case 
when $\gamma\in (-1,0)$ ({\em i.e.}, $s\in (3,5)$),
and {\em very soft potentials} the case 
when $\gamma\in (-3,-1)$ ({\em i.e.}, $s\in (2,3)$).

\vip

In any case, $\int_{0+}  \beta(\theta) d\theta = + \infty$,
which expresses the affluence of {\em grazing collisions}, that
is collisions with a very small deviation.
We will assume here the general physically reasonnable
conditions
\renewcommand\theequation{{\bf A2}}
\begin{equation}
\int_0^\pi \beta(\theta)d\theta	= +\infty, \quad 
\kappa_1 :=\int_0^\pi \theta^2 \, \beta(\theta)d\theta < + \infty.
\end{equation}
\renewcommand\theequation{\thesection.\arabic{equation}}

We now introduce, for 
$\theta \in (0,\pi]$, 
\begin{equation}\label{defH}
H(\theta):=\int_\theta^\pi \beta(x)dx \quad \hbox{ and } \quad G(z):=H^{-1}(z).
\end{equation}
Here $H$ is a continuous deacreasing bijection from $(0,\pi]$
into $[0,+\infty)$, and its inverse function $G:[0,+\infty)
\mapsto (0,\pi]$ is defined by $G(H(\theta))=\theta$, and $H(G(z))=z$.
We will suppose that there exists $\kappa_2>0$ such that
for all $x, y \in \rr_+$,
\renewcommand\theequation{{\bf A3}}
\begin{equation}
\int_0^\infty \left(G(z/x)-G(z/y) \right)^2 dz  
\leq \kappa_2 \frac{(x-y)^2}{x+y}.
\end{equation}

Concerning the velocity part of the cross section, 
we will assume that for all $x,y \in \rr_+$,

\renewcommand\theequation{\thesection.\arabic{equation}}

\begin{eqnarray}\label{a4}
\min(x^2,y^2) \frac{[\Phi(x)-\Phi(y)]^2}{\Phi(x)+\Phi(y)}
+(x-y)^2[\Phi(x)+\Phi(y)]&&\ala
+ \min(x,y) |x-y| |\Phi(x)-\Phi(y)|
&\leq& (x-y)^2 [ \Psi(x)+\Psi(y)].
\end{eqnarray}

for some function $\Psi:\rr_+\mapsto\rr_+$, with 
for some $\gamma \in (-3,0]$, some $\kappa_3>0$, for all $x\in\rr_+$,

\renewcommand\theequation{{\bf A4}$(\gamma)$}
\begin{equation}
\Psi(x) \leq \kappa_3 x^\gamma.
\end{equation}

\renewcommand\theequation{\thesection.\arabic{equation}}

Under Assumption ({\bf (A4)}$(\gamma)$), we can easy see that necessarily for all $x\in\rr_+$, $\Phi(x)\leq \Psi(x)$, and then $\Phi(x)\leq \kappa_3 x^\gamma$.

These assumptions are not very transparent.
However, the following lemma, proved in the appendix, 
shows how they apply. Roughly,
{\bf (A3)} is very satisfying,
{\bf (A4$(0)$)} corresponds to {\it regularized} 
velocity cross sections,
while {\bf (A4$(\gamma)$)} allows us to deal with general soft potentials.

\begin{lem}\label{raiso}
(i) Assume that there are $0<c<C$ and $\nu \in (0,2)$ such that 
for all $\theta\in (0,\pi]$, $c \theta^{-\nu-1}\leq 
\beta(\theta) \leq C\theta^{-\nu-1}$. Then {\bf (A2-A3)} hold.

(ii) Assume that 
$\Phi(x)=\min(x^\alpha,A)$ for some $A>0$, some $\alpha \in \rr$, or that $\Phi(x)=(\e+x)^\alpha$
for some $\e>0$, $\alpha< 0$. Then {\bf (A4$(0)$)} holds.

(iii) Assume that for some $\gamma \in (-3,0]$, $\Phi(x)=x^\gamma$.
Then {\bf (A4$(\gamma)$)} holds.
\end{lem}

\renewcommand\theequation{\thesection.\arabic{equation}}

\subsection{Goals, existing results and difficulties}

We study in this paper the well-posedness of the spatially 
homogeneous Boltzmann equation for singular collision kernel as 
introduced above. In particular we are interested in uniqueness 
and stability with respect to the initial condition.

\vip

In the case of a collision kernel with angular cutoff, that is when 
$\int_0^\pi \beta( \theta) d\theta<+ \infty$, 
there are some optimal existence and 
uniqueness results: see Mischler-Wennberg \cite{mw} and Lu-Mouhot \cite{lumou}.

\vip

The case of collision kernels without cutoff is much more 
difficult, but is very important, since it corresponds to the previously
described physical collision kernels.
This difficulty is not surprising:
on each compact time interval, each particle collides with infinitely 
(resp. finitely) many others 
in the case without (resp. with) cutoff.

\vip

In all the previously cited physical situations, global existence of
weak solutions has been proved by Villani \cite{villexi} by using some
compactness methods.

\vip

Until recently, the only uniqueness result obtained 
for non cutoff 
collision kernel was concerning Maxwellian molecules, studied
successively by Tanaka \cite{tanaka}, Horowitz-Karandikar \cite{kh},
Toscani-Villani \cite{tv}: it was proved in \cite{tv} 
that uniqueness holds for the Boltzmann equation 
as soon as
$\Phi$ is constant and {\bf (A2)} is met, for any 
initial (measure) datum with finite mass and energy, that is
$\int_{\rd} (1+|v|^2)\, f_0(dv)< +\infty$.

\vip

There has been recently three papers in the case where $\beta$ is non 
cutoff and 
$\Phi$ is not constant. The case where $\Phi$ is bounded (together 
with additionnal 
regularity assumptions) was treated in \cite{f}, for 
essentially any initial (measure) datum such that
$\int_{\rr^d} (1+|v|)f_0(dv)<\infty$. 
More realistic collision kernels have been treated 
by Desvillettes-Mouhot \cite{dm} and Fournier-Mouhot \cite{fmou}
(including hard and moderately soft potentials).
However, all these results apply only when assuming the following
condition, stronger than {\bf (A2)}, 
\begin{equation}
\int_0^\pi \theta \beta(\theta)d\theta<\infty.
\end{equation}
In particular, this does not apply to very soft potentials ($s\in (2,3]$).
Weighted Sobolev spaces were used in \cite{dm}, while
the results of \cite{fmou}
rely on the Kantorovich-Rubinsten distance.

\vip

In the present paper, we obtain the first uniqueness result
which can deal with the case where only {\bf (A2)} is supposed.
Our result is based on the use of the Wasserstein distance
with quadratic cost. 
The main interest of our paper concerns very soft potentials,
for which we obtain the uniqueness
of the solution provided it remains in $L^p(\rd)$,
for some $p$ large enough. Since we are only able
to propagate locally such a property, we obtain some
local (in time) well-posedness result. 

\vip

Our method certainly applies to the case of hard potentials. We
however do not treat this case in the present paper, since there
are already some available uniqueness results, as said previously.
Let us also mention that in a companion paper, we use a similar
method to get some uniqueness result for the Landau equation
with soft potentials, which was still open.

\vip 

Our proof is probabilistic, and we did not manage
to rewrite it in an analytic way. The main idea is quite
simple: for two solutions $(f_t)_{t\geq 0}$, $(\tf_t)_{t\geq 0}$
to the Boltzmann equation, we construct two stochastic processes
$(V_t)_{t\geq 0}$ and $(\tV_t)_{t\geq 0}$ whose time marginal laws are
given by $(f_t)_{t\geq 0}$ and $(\tf_t)_{t\geq 0}$, and which are
coupled in such a way that $E[|V_t-\tV_t|^2]$ is ``small'' for all times.
This bounds from above the Wasserstein distance with quadratic cost
between $f_t$ and $\tf_t$.

\subsection{Notation} 
Let us denote by $C^2_\infty$ (resp. $C^2_b$, resp. $C^2_c$) 
the set of $C^2$-functions 
$\phi:\rd \mapsto \rr$ of which
the second derivative is bounded (resp. of which
the derivatives of order $0$ to $2$ are bounded, resp.
which are compactly supported). 

Let also $L^p(\rd)$ be the space of measurable 
functions $f$ with
$\| f \|_{L^p(\rd)} := \left( \int_{\rd} f^p (v)\, dv \right)^{1/p} < +\infty$.

Let $\cP(\rd)$ be the set of probability measures on 
$\rd$, and 
\begin{equation*}
\pdeux = \left\{f \in \cP(\rd), \; m_2(f) <\infty \right\} \quad 
\mbox{ with } \quad  
m_2(f) := \intrd |v|^2 \, f(dv).
\end{equation*}

For $\alpha \in (-3,0]$, we introduce the space 
$\cJ_\alpha$ of probability measures $f$ on $\rd$ such that
\begin{eqnarray}\label{cjalpha}
J_\alpha(f)&:=&\sup_{v\in\rd} \intrd |v-v_*|^\alpha f(dv_*) <\infty.
\label{cj0}
\end{eqnarray}
Of course, for any probability measure $f$, $J_0(f)=1$.
Let $\lipdeux$, $\lunlp$ and $L^1([0,T],\cJ_\alpha)$
be the sets of 
measurable families $(f_t)_{t\in [0,T]}$ of 
probability measures on $\rd$ with
$\sup_{[0,T]} \, m_2(f_t) < + \infty$,
$\int_0 ^T \| f_t \|_{L^p(\rd)} \, dt <+ \infty$, and
$\int_0 ^T J_\alpha(f_t) dt <+ \infty$ respectively.

\subsection{Weak solutions}

We follow here \cite{fm}.
For each $X\in \rd$, we introduce $I(X),J(X)\in\rd$ such that
$(\frac{X}{|X|},\frac{I(X)}{|X|},\frac{J(X)}{|X|})$ 
is an orthonormal basis of $\rd$. 
We also require that $I(-X)=-I(X)$ and $J(-X)=-J(X)$ for convenience.
For $X,v,v_*\in \rd$, for $\theta \in [0,\pi]$ and $\varphi\in[0,2\pi)$,
we set
\begin{equation}\label{dfvprime}
\left\{
\begin{array}{l}
\Gamma(X,\varphi):=(\cos \varphi) I(X) + (\sin \varphi)J(X), \\
v':=v'(v,v_*,\theta,\varphi):= v - \frac{1-\cos\theta}{2} (v-v_*)
+ \frac{\sin\theta}{2}\Gamma(v-v_*,\varphi),\\
v'_*:=v'_*(v,v_*,\theta,\varphi):= v_* + \frac{1-\cos\theta}{2} (v-v_*)
- \frac{\sin\theta}{2}\Gamma(v-v_*,\varphi),\\
a:=a(v,v_*,\theta,\varphi):=(v'-v)=-(v'_*-v_*),
\end{array}
\right.
\end{equation}
which is nothing but a suitable spherical parameterization of (\ref{vprimeetc}):
we write $\sigma\in \Sp^2$ as $\sigma=\frac{v-v_*}{|v-v_*|}\cos\theta
+ \frac{I(v-v_*)}{|v-v_*|}\sin\theta \cos\varphi+\frac{J(v-v_*)}{|v-v_*|}\sin\theta
\sin\varphi$.

Let us observe at once that
\begin{eqnarray}
v'(v_*,v,\theta,\varphi)=v'_*(v,v_*,\theta,\varphi),&&
v'_*(v_*,v,\theta,\varphi)=v'(v,v_*,\theta,\varphi),   \label{obs1}\\
v'(v,v_*,\pi-\theta,\varphi)=v'_*(v,v_*,\theta,\varphi+\pi),&&
v'_*(v,v_*,\pi-\theta,\varphi)=v'(v,v_*,\theta,\varphi+\pi).\label{obs2}
\end{eqnarray}
Let us define the notion of weak solutions we shall use.

\begin{defin}\label{dfsol}
Let $B$ be a collision kernel which satisfies {\bf (A1-A2)}. 
A family $f=(f_t)_{t\in [0,T]} \in \lipdeux$ 
is a weak solution to (\ref{be}) if
\begin{equation}\label{concon}
\int_0^T dt  \intrd f_t(dv)\intrd f_t(dv_*) \, \Phi(|v-v_*|) \, |v-v_*|^2
< + \infty,
\end{equation}
and if for any $\phi \in C^2_\infty$, and any $t\in[0,T]$,
\begin{equation}\label{wbe}
\intrd \phi(v)\, f_t(dv) =  \intrd \phi(v)\, f_0(dv)
+\intot ds \intrd f_s(dv) \intrd f_s(dv_*) \, \cA\phi(v,v_*),
\end{equation}
where
\begin{equation}\label{afini}
\cA\phi (v,v_*) = \frac{\Phi (|v-v_*|)}{2} \, \int_0^\pi 
\beta(\theta)d\theta 
\int_0^{2\pi} d\varphi
\left[\phi(v')+\phi(v'_*)-\phi(v)-\phi(v_*) \right].
\end{equation}
\end{defin}

As noted by Villani \cite[p 291]{villexi}, one has, 
for all $v,v_* \in \rd$, all $\theta\in[0, \pi]$, all
$\phi \in C^2_\infty$,
\begin{equation}\label{majfond} 
\left| \int_0^{2\pi}d\varphi \left[\phi(v')+\phi(v'_*)-\phi(v)-\phi(v_*) 
\right]  \right| \leq C ||\phi''||_\infty \theta^2 |v-v_*|^2,
\end{equation}
so that thanks to assumption {\bf (A2)}, (\ref{concon}) ensures that  
all the terms in (\ref{wbe}) are well-defined.
The proof of (\ref{majfond}) is given in the appendix for the sake of 
completeness.

\vip 

\subsection{A suitable distance} \label{dfdist}
Let us now introduce the distance on $\pdeux$ we shall use. 
For  $g,\tg \in \pdeux$, let $\cH(g,\tg)$ be the set of probability 
measures on $\rd\times\rd$ with first marginal $g$ and second marginal $\tg$.
We then set
\begin{eqnarray}\label{dfw2}
W_2(g,\tg)&=&
\inf \left\{\left(\intrdd |v-\tv|^2 \, G(dv,d\tv)\right)^{1/2},\quad G\in \cH(g,\tg)\right\}\ala
&=&\min\left\{\left(\intrdd |v-\tv|^2\, G(dv,d\tv)\right)^{1/2},\quad G\in \cH(g,\tg)\right\}.
\end{eqnarray}
This distance is the so-called 
Wasserstein distance with quadratic cost.
We refer to Villani \cite[Chapter 2]{villani2} 
for more details on this distance.

\vip 

Our result is based on the use of this distance. A remarkable result,
due to Tanaka \cite{tanaka}, is that in the Maxwellian case, that is when 
$\Phi\equiv 1$, $t\mapsto W_2(f_t,\tf_t)$ is nonincreasing 
for each pair of reasonnable solutions $f,\tf$ to the Boltzmann equation.

\subsection{The main results}
Our main result is the following inequality.

\begin{theo}\label{maintheo}
Assume {\bf (A1-A2-A3-A4$(\gamma)$)}
for some $\gamma\in (-3,0]$. 
Let us consider two weak solutions  $(f_t)_{t\in[0,T]},(\tf_t)_{t\in[0,T]}$ 
to (\ref{be}) lying in $\lipdeux \cap L^1([0,T],\cJ_\gamma)$.
Assume furthermore than for all $t\in [0,T]$, $f_t$ (or $\tf_t$) has
a density with respect to the Lebesgue measure on $\rd$.
There exists a constant $K=K(\kappa_1,\kappa_2,\kappa_3)$ such that
for all $t \in [0,T]$,
\begin{eqnarray}\label{resultat}
W_2(f_t,\tf_t) \leq W_2(f_0,\tf_0) \exp \left(
K\int_0^t J_\gamma(f_s+\tf_s) ds \right).
\end{eqnarray}
\end{theo}

Observe here that the technical assumption that $f_t$ has a density can
easily be removed, provided one has some uniform estimates on $J_\gamma(f_t)$,
as will be the case in the applications below.

\vip

We first 
give some application to the case of mollified velocity cross sections.

\begin{cor}\label{nul}
Assume {\bf (A1-A2-A3-A4$(0)$)}.
For any $f_0\in \pdeux$, any $T>0$, there exists a unique weak solution 
$(f_t)_{t\in[0,T]}\in\lipdeux$ to  (\ref{be}). Furthermore, there
exists a constant $K=K(\kappa_1,\kappa_2,\kappa_3)$ such that for any
pair of weak solutions $(f_t)_{t\in[0,T]}$ and $(\tf_t)_{t\in[0,T]}$
to (\ref{be}) in $\lipdeux$, 
any $t>0$,
\begin{eqnarray}\label{resultatnul}
W_2(f_t,\tf_t) \leq W_2(f_0,\tf_0) e^{K t}.
\end{eqnarray}
\end{cor}

We now apply our inequality to the case of soft potentials.

\begin{cor}\label{sp}
Assume {\bf (A1-A2-A3-A4$(\gamma)$)}
for some $\gamma \in (-3,0]$, and let $p\in (3/(3+\gamma),\infty)$.

(i) For any pair
of weak solutions 
$(f_t)_{t\in[0,T]}$, $(\tf_t)_{t\in[0,T]}$ to (\ref{be}) 
lying in $L^\infty([0,T],\pdeux) \cap L^1([0,T], L^p(\rd))$, there holds
\begin{equation}
\forall \, t \in [0,T], \qquad W_2(f_t,\tf_t) \leq W_2(f_0,\tf_0) \, 
\exp \left(K_p \int_0^t [1+ \| f_s \|_{L^p(\rd)}+  \| \tf_s \|_{L^p(\rd)} ]
ds  \right)
\end{equation}
where $K_p$ depends only on 
$\gamma,p,\kappa_1,\kappa_2,\kappa_3$.
Uniqueness and stability with respect to the initial condition thus 
hold in $L^\infty([0,T],\pdeux) \cap L^1([0,T], L^p(\rd))$.

(ii) For any
$f_0 \in \pdeux \cap L^p(\rd)$, 
there exists 
$T_*=T_* \big(\|f_0\|_{L^p(\rd)},p,\gamma,\kappa_1,\kappa_2,\kappa_3 \big)>0$
such that there exists a unique weak solution $(f_t)_{t \in [0,T_*)}$ 
to (\ref{be}) lying in  
$L^\infty_{\mbox{{\scriptsize {\em loc}}}}\big([0,T_*),\pdeux\cap 
L^p(\rd)\big)$.
\end{cor}

The rest of the paper is dedicated to the proof of these results.
We first state some preliminary lemmas in Section \ref{prel}.
Since the rigorous proof of Theorem \ref{maintheo}, handled in Section \ref{secmain},
is quite complicated, we first give some formal arguments in Section
\ref{notrig}.
Corollaries \ref{nul} and \ref{sp} are checked in Section \ref{seccor}.
Finally, an appendix containing
technical computations lies at the end of the paper.

\section{Preliminaries}\label{prel}\setcounter{equation}{0}

We start by a suitable way to rewrite the collision
operator. The main interest of the following expression is that 
we make disappear the velocity-dependance
$\Phi(|v-v_*|)$ in the {\it rate}. Such a trick 
was already used in \cite{fm2}.

\begin{lem}\label{rewriteA} 
Assume {\bf (A1-A2)} and set 
\begin{equation}\label{dfkzero}
\kappa_0:=\pi 
\int_0^\pi (1-\cos\theta) \beta(\theta)d\theta. 
\end{equation}
Recalling 
(\ref{defH}) and (\ref{dfvprime}), define,
for $z\in (0,\infty)$,
$\varphi\in [0,2\pi)$, $v,v_*\in \rd$, 
\begin{equation}\label{dfc}
c(v,v_*,z,\varphi):=a[v,v_*,G(z/\Phi(|v-v_*|)),\varphi].
\end{equation} 
We have $\cA\phi(v,v_*)= \frac{1}{2}[\ctA\phi(v,v_*) + \ctA\phi(v_*,v)]$
for all $v,v_*\in\rd$ and $\phi\in C^2_\infty$, where
\begin{eqnarray}\label{agood}
\ctA\phi(v,v_*)&=&\int_0^\infty dz \int_0^{2\pi}d\varphi
\Big( \phi[v+c(v,v_*,z,\varphi)] -\phi[v] - 
c[v,v_*,z,\varphi].\nabla\phi[v] \Big) \ala
&& -\kappa_0 \Phi(|v-v_*|) \nabla\phi(v).(v-v_*) \ala
&=&\int_0^\infty dz \int_0^{2\pi}d\varphi
\Big(\phi[v+c(v,v_*,\varphi+\varphi_0)] -\phi[v]
- c[v,v_*,z,\varphi+\varphi_0].\nabla\phi[v] \Big) \ala
&& -\kappa_0 \Phi(|v-v_*|) \nabla\phi(v).(v-v_*),
\end{eqnarray}
the second equality holding for any $\varphi_0\in [0,2\pi)$ (which
may depend on $v,v_*,z$).
As a consequence, we may replace $\cA$ by $\ctA$ in (\ref{wbe}).
\end{lem}

This lemma is proved in the appendix.
Let us now recall a fundamental remark by Tanaka \cite{tanaka},
slighlty precised in \cite[Lemma 2.6]{fm}.
We use here notation (\ref{dfvprime}).

\begin{lem}\label{tanana}
There exists a measurable function $\varphi_0 : \rd \times \rd
\mapsto [0,2\pi)$, such that for all $X,Y\in\rd$, all $\varphi\in[0,2\pi)$,
\begin{equation}
|\Gamma(X,\varphi) - \Gamma(Y, \varphi + \varphi_0(X,Y))|
\leq 3 |X - Y|.
\end{equation}
\end{lem}

The following fundammental estimates, on which our results
rely, are proved in the appendix.

\begin{lem}\label{fundest}
Assume {\bf (A1-A2-A3)} and (\ref{a4}). There exists a constant
$C=C(\kappa_1,\kappa_2)$ such that the following estimates hold.

(i) For $v,v_*,\tv,\tv_* \in \rd$,
\begin{eqnarray}
&&\int_0^\infty dz \int_0^{2\pi} d\varphi |c(v,v_*,z,\varphi)|^2 \leq 
C|v-v_*|^{2} \Phi(|v-v_*|), \label{esti1} \\
&&\int_0^\infty dz \int_0^{2\pi} d\varphi |c(v,v_*,z,\varphi)
- c(\tv,\tv_*,z,\varphi+\varphi_0(v-v_*,\tv-\tv_*))|^2 \label{esti2} \\
&&\hskip3cm 
\leq C (|v-\tv|^2+|v_*-\tv_*|^2) (\Psi(|v-v_*|)+\Psi(|\tv-\tv_*|)),\ala
&& \int_0^\infty dz \left|\int_0^{2\pi} d\varphi \; c(v,v_*,z,\varphi) 
\right| \leq 
C|v-v_*| \Phi(|v-v_*|), \label{esti3} \\
&& \int_0^\infty dz \left| \int_0^{2\pi} d\varphi [c(v,v_*,z,\varphi)
- c(\tv,\tv_*,z,\varphi)] \right| \label{esti4} \\
&&\hskip3cm 
\leq C (|v-\tv|+|v_*-\tv_*|) (\Psi(|v-v_*|)+\Psi(|\tv-\tv_*|)). \nonumber
\end{eqnarray}

(ii) For any $\phi \in C^2_b$, any $v,v_* \in \rd$,
\begin{equation}\label{atfini1}
|\ctA \phi (v,v_*)| \leq C \Phi(|v-v_*|)\left(|v-v_*| . ||\phi'||_\infty + 
|v-v_*|^2 ||\phi''||_\infty\right).
\end{equation}

(iii) For any $\phi \in C^2_c$ with {\rm supp} $\phi \subset \{|v|\leq x\}$,
for all $v,v_*\in \rd$,
\begin{eqnarray}\label{atfini2}
|\ctA \phi (v,v_*)| &\leq& C \left( ||\phi'||_\infty |v-v_*| 
+||\phi''||_\infty 
|v-v_*|^2 \right) \Phi(|v-v_*|) \indiq_{\{|v|\leq 2x\}}  \ala
&&+ C ||\phi||_\infty \frac{|v-v_*|^2 \Phi(|v-v_*|)}{|v|^2} 
\indiq_{\{|v|\geq 2x\}}.
\end{eqnarray}
\end{lem}

We now state again some estimates that will be usefull when 
passing to the limit
in some cutoff Boltzmann equations.

\begin{lem}\label{moinsfundest}
We assume {\bf (A1-A2-A3)} and (\ref{a4}). For $k\geq 1$ and $x\in\rr_+$, 
we set
\begin{equation}\label{dfhzero}
h_0^k(x):=\pi \int_0^k dz [1- \cos(G[z/\Phi(x)])] \hbox{ and } 
\e_0^k(x) := \int_0^{G[k/\Phi(x)]} \theta^2\beta(\theta) d\theta
\end{equation}
There exists a constant $C=C(\kappa_1,\kappa_2)$ such that for all
$v,v_* \in \rd$, all $x,y\in \rr_+$, all $k \geq 1$,
\begin{eqnarray}
&&\int_k^\infty dz \intzdp d\varphi |c(v,v_*,z,\varphi)|^2 \leq 
C|v-v_*|^{2} \Phi(|v-v_*|) \e_0^k(|v-v_*|) \label{mesti1} \\
&& | xh_0^k(x) - yh_0^k(y)|\leq C |x-y|(\Psi(x)+\Psi(y)),
\label{mesti2} \\
&& | \kappa_0 x \Phi(x) - x h_0^k(x) | \leq C x \Phi(x) \e_0^k(x). 
\label{mesti3}
\end{eqnarray}
Furthermore, $\e_0^k$ is bounded by $\kappa_1$, and for all $x\in \rr_+$,
$\lim_k \e_0^k(x)=0$.
\end{lem}

This Lemma will be checked in the appendix,
as the following continuity property of $\ctA$.

\begin{lem}\label{acont}
Assume {\bf (A1-A2-A3-A4)}$(\gamma)$, for
some $\gamma\in (-3,0]$, and consider $g\in \pdeux\cap \cJ_\gamma$. 
Then for any $\phi \in C^2_c$, $v\mapsto \int_{\rd} g(dv_*) \ctA\phi(v,v_*)$
is continuous on $\rd$.
\end{lem}

\section{A short and unrigorous proof}\label{notrig}

We give here the main idea of this paper in the cutoff case. 
In the case without cutoff, we 
are not able to give a direct proof (not relying on the
use of Poisson measures, martingale problems,... see the next section).
We consider a solution $(f_t)_{t\in[0,T]}$ 
to the Boltzmann equation. Then 
\begin{equation}
\frac{d}{dt} \intrd \phi(v) f_t(dv) = \intrd f_t(dv) \intrd f_t(dv_*)
\ctA\phi(v,v_*), 
\end{equation}
and 
we can formally 
write
\begin{equation}
\ctA \phi(v,v_*) = 2\pi\int_0^\infty dz \intzdp \frac{d\varphi}{2\pi}
[\phi(v+c(v,v_*,z,\varphi)) -\phi(v)].
\end{equation}
This roughly means the following: take a particle at random at time $t$,
and call its velocity $V_t$. Then $V_t$ is $f_t$-distributed. Then
for all $z\in\rr_+$, it will collide, at rate $2\pi dz$,  with another particle
with velocity $V_t^*$ (independant and also $f_t$-distributed), it will
choose $\alpha$ uniformly in $[0,2\pi)$, and its new velocity after
the collision will be $Z_t(z):=V_t+c(V_t,V_t^*,z,\alpha)$. Let us call 
$\Delta(f_t,z)$ the law of $Z_t(z)$.

Then if we have two solutions $(f_t)_{t\in[0,T]},(\tf_t)_{t\in[0,T]}$
to (\ref{be}), it is natural to think that
\begin{eqnarray}\label{superw}
\frac{d}{dt} W^2_2(f_t,\tf_t) \leq \int_0^\infty 2\pi dz [W_2^2(\Delta(f_t,z),\Delta(\tf_t,z)) - W^2_2(f_t,\tf_t)].
\end{eqnarray}
Indeed, at each time $t$, for each $z$, $f_t$ and $\tf_t$ are replaced by
$\Delta(f_t,z)$ and $\Delta(\tf_t,z)$ at rate $2\pi dz$.

\vip

Such an inequality can be rigorously and easily obtained
when truncating the integral $\int_0^\infty dz$ into $\int_0^k dz$,
by using the dual formulation of the Wasserstein distance (see Villani
\cite{villani2}).

\vip

We then claim that for all pair of laws $f,\tf$ on $\rd$,
\begin{eqnarray}\label{afaire}
\int_0^\infty 2\pi dz [W_2^2(\Delta(f,z),\Delta(\tf,z)) - W^2_2(f,\tf)]
\leq C W_2^2(f,\tf) [J_\gamma(f)+J_\gamma(\tf)],
\end{eqnarray}
where $C$ depends only on $\kappa_1,\kappa_2,\kappa_3$, see 
{\bf (A1-A2-A3-A4)}$(\gamma)$.
Gathering (\ref{superw}) and (\ref{afaire}), Theorem \ref{maintheo}
would follow immediately from the generalized Gronwall Lemma \ref{ggl}.

Let us prove (\ref{afaire}). Consider thus $f,\tf$ two probability distributions
on $\rd$, and two couples $(V,\tV)$ and $(V_*,\tV_*)$
with $V$ and $V_*$ $f$-distributed, $\tV$ and $\tV_*$ $\tf$-distributed, with
$(V,\tV)$ independent of $(V_*,\tV_*)$, and such that 
$E[|V-\tV|^2]=E[|V_*-\tV_*|^2]=W_2^2(f,\tf)$.
Choose $\alpha$ uniformly distributed on $[0,2\pi)$ (independent of everything else),
and set $\talpha=\alpha+\varphi_0(V-V_*,\tV-\tV_*)$ (modulo $2\pi$), where
$\varphi_0$ was introduced in Lemma \ref{tanana}.
Then $\talpha$ is also uniformly distributed on $[0,2\pi)$, and is also
independent of $(V,\tV,V_*,\tV_*)$. As a consequence,
$Z(z)=V+c(V,V_*,z,\alpha)$ is $\Delta(f,z)$-distributed, and 
$\tZ(z)=\tV+c( \tV,\tV_*,z,\talpha)$ is $\Delta(\tf,z)$-distributed, so that
\begin{equation}
W^2_2(\Delta(f,z),\Delta(\tf,z))-W^2_2(f,\tf)\leq E[|Z(z)-\tZ(z)|^2-|V-\tV|^2]
=:\delta(z).
\end{equation}
But a simple computation using (\ref{esti2}) and (\ref{esti4}) 
shows that for some constant
$C=C(\kappa_1,\kappa_2,\kappa_3)$,
\begin{eqnarray}
\int_0^\infty dz \delta(z)&=&\int_0^\infty dz \int_0^{2\pi} \frac {d\varphi}{2\pi}
E\Big[| c(V,V_*,z,\varphi)-c(\tV, \tV_*,z,\varphi+\varphi_0(V-V_*,\tV-\tV_*))|^2 \ala
&&\hskip1cm
+ 2 (V-\tV)(c(V,V_*,z,\varphi)-c(\tV,\tV_*,z,\varphi+\varphi_0(V-V_*,\tV-\tV_*))) \Big]\ala
&&\leq C E\left[(|V-\tV|^2+|V_*-\tV_*|^2 ) (\Psi(|V-V_*|)+\Psi(|\tV-\tV_*|) ) \right]
\ala
&&\leq C E\left[|V-\tV|^2 (|V-V_*|^\gamma+|\tV-\tV_*|^\gamma ) \right]
\end{eqnarray}
by a symmetry argument and {\bf (A4)}$(\gamma)$. Using finally 
the definition of $J_\gamma$ and the independance of $(V,\tV)$ and $(V_*,\tV_*)$,
one easily deduces that
\begin{eqnarray}
\int_0^\infty dz \delta(z)&\leq&
C E\left[|V-\tV|^2\right] 
(\sup_v E[|v-V_*|^\gamma]+\sup_{\tv}E[|\tv-\tV_*|^\gamma] )\ala
&=&C W_2^2(f,\tf)[J_\gamma(f)+
J_\gamma(\tf)].
\end{eqnarray}
This concludes the proof of (\ref{afaire}).

\section{Coupling Boltzmann processes}
\label{secmain}\setcounter{equation}{0}

To prove Theorem \ref{maintheo}, we will use
some probabilistic arguments,
which is of course a natural way to couple two solutions of the Boltzmann
equation.
We follow the line of Tanaka \cite{tanaka} (see also \cite{fm}),
who was dealing with the Maxwellian case, that is $\Phi\equiv 1$.

\vip

In the whole section, $C$ (resp. $C_T$) stands for a constant
whose value may change from line to line, and which depend only on 
$\kappa_1,\kappa_2,\kappa_3,\gamma$ 
(resp. $\kappa_1,\kappa_2,\kappa_3,\gamma,T$).

\vip

Recall that $\dt=\mathbb{D}([0,T],\rd)$ stands for the Skorokhod space
of c\`adl\`ag functions, see Ethier-Kurtz 
\cite{ek} for many details on this topic. 
We consider a filtered probability space $(\Omega,\cF,(\cF_t)_{t\in[0,T]},P)$.

\begin{nota}\label{notagpro}
Let $g=(g_t)_{t\in [0,T]}$ be a measurable family of probability
measures on $\rd$.

(i) We say that $N$ is a $g$-Poisson measure
if it is a $(\cF_t)_{t\in[0,T]}$-Poisson measure on $[0,T]\times \rd \times
[0,\infty)\times [0,2\pi)$ with intensity measure 
$ds g_s(dv) dz d\varphi$. We denote by $\tN$ its compensated Poisson
measure. 

(ii) For $k\geq 1$, $V_0$ a $\cF_0$-measurable $\rd$-valued random variable
and $N$ a $g$-Poisson measure, we define
$(V^k_t)_{t\in[0,T]}$ the unique solution to
\begin{eqnarray}\label{dfvk}
V^k_t&=&V_0 + \intot\intrd\int_0^k\int_0^{2\pi} c(V^k_\sm,v,z,\varphi)
N(ds,dv,dz,d\varphi).
\end{eqnarray}
Then $(V^k_t)_{t\in[0,T]}$ is adapted to $(\cF_t)_{t\in[0,T]}$ 
and belongs a.s. to $\dt$. We will refer to $(V^k_t)_{t\in[0,T]}$
as the $(V_0,g,k,N)$-process. 
Its law does not depend on the choice of the probability space, on $N$,
and depends on $V_0$ only through its law.
\end{nota}

The existence and uniqueness of $V^k$ is obvious, because
$N([0,T]\times\rd\times[0,k]\times[0,2\pi))$ is a.s. finite,
so that (\ref{dfvk}) is nothing but a recursive equation. 

We will show the following result at the end of this section.

\begin{lem}\label{cacv}
Assume {\bf (A1-A2-A3-A4$(\gamma)$)}, for some $\gamma \in (-3,0]$.
Consider a weak solution $(f_t)_{t\in[0,T]} \in \lipdeux\cap
 L^1([0,T],\cJ_\gamma)$ 
to (\ref{be}).
Consider any ${\cF_0}$-measurable random variable $V_0\sim f_0$.
Consider a $f$-Poisson measure $N$, and for each $k\geq 1$,
the $(V_0,f,k,N)$-process $(V^k_t)_{t\in [0,T]}$.
For each $t\in [0,T]$, denote by $f^k_t$ the law of $V^k_t$.
Then
\begin{equation}
\lim_{k\to \infty} \sup_{[0,T]} W_2^2 (f_t,f^k_t) =0.
\end{equation}
\end{lem}

Thus we will study a solution $f$ to (\ref{be}) through its 
related stochastic process $V^k_t$.
We start with some moment computations.

\begin{lem}\label{m2k}
Assume {\bf (A1-A2-A3-A4)$(\gamma)$} for some $\gamma\in (-3,0]$. 
Let $g\in \lipdeux\cap L^1([0,T],\cJ_\gamma)$.
There exists a constant $K_T(g)$ depending only on 
$T,\gamma,\kappa_1,\kappa_2,\kappa_3,g$ such that
for each $k\geq 1$, $V_0\in L^2$,
each $g$-Poisson measure $N$, the $(V_0,g,k,N)$-process $(V^k_t)_{t\in [0,T]}$
satisfies
\begin{equation}
E \left[\sup_{[0,T]} |V_t^k|^2 \right] 
\leq K_T(g) \{1+E[|V_0|^2]\}.
\end{equation}
\end{lem}

\begin{proof}
Let $k\geq 1$ be fixed. Writing the Poisson measure $N$
as $\tN+ ds g_s(dv)dzd\varphi$, 
we obtain, using the Doob inequality, 
that for $t\in [0,T]$,  
$E[\sup_{[0,t]} |V_s^k|^2] \leq C\{E[|V_0|^2] + A_t + B_t \}$, where
\begin{eqnarray}
A_t & := & E \left[\intot ds \intrd g_s(dv) \int_0^k dz \int_0^{2\pi} d\varphi 
|c(V^k_s,v,z,\varphi)|^2 \right], \ala
B_t & := & E \left[ \sup_{[0,t]} \left|\intot ds \intrd g_s(dv) 
\int_0^k dz \int_0^{2\pi} d\varphi c(V^k_s,v,z,\varphi) \right|^2 \right]. 
\end{eqnarray}
Using now (\ref{esti1}) and then {\bf (A4$(\gamma)$)}, we get
\begin{eqnarray}
A_t  &\leq&   C E \left[\intot ds \intrd g_s(dv) 
|V^k_s - v|^{2+\gamma}\right],
\end{eqnarray}
while (\ref{esti3}) and  {\bf (A4$(\gamma)$)} yield
\begin{eqnarray}
B_t  &\leq & C E \left[ \left|\intot ds \intrd g_s(dv) 
|V^k_s - v|^{\gamma+1}\right|^2 \right].
\end{eqnarray}
We then have to divide the study into several cases.

\vip

{\it Case $\gamma \in [-1,0]$.} Then $\gamma+2 \in [0,2]$, so that
$|V^k_s - v|^{2+\gamma} \leq C(1+|V^k_s|^2+|v|^2)$, and one easily checks that
$A_t \leq C_T (1+\int_0^t m_2(g_s) ds + \int_0^t E[|V^k_s|^2]ds)$.
Furthermore $\gamma+1 \in [0,1]$, so that $|V^k_s - v|^{1+\gamma} 
\leq C(1+|V^k_s|+|v|)$. Thus $B_t \leq 
C_T (1+\int_0^t m_2(g_s) ds + \int_0^t E[|V^k_s|^2]ds)$
by the Cauchy-Schwarz inequality.
We finally find 
$E[\sup_{[0,t]} |V_s^k|^2] \leq 
C_T( 1+E[|V_0|^2]+\int_0^t m_2(g_s) ds + \int_0^t E[|V^k_s|^2]ds)$
and the conclusion follows by the Gronwall Lemma, since
$\int_0^T m_2(g_s) ds < \infty$ by assumption.

\vip

{\it Case $\gamma \in [-2,-1]$.} Since $\gamma+2 \in [0,2]$, we obtain
as previously $A_t \leq C (1+\int_0^t m_2(g_s) ds + \int_0^t E[|V^k_s|^2]ds)$.
On the other hand, $\gamma<\gamma+1\leq 0$, so that 
$|V^k_s-v|^{\gamma+1} \leq 1 + |V^k_s-v|^\gamma$. Recalling (\ref{cjalpha}),
we deduce that $\int_{\rd} g_s(dv) |V^k_s-v|^{\gamma+1}
\leq 1 + \int_{\rd}g_s(dv) |V^k_s-v|^{\gamma} 
\leq  1 + J_\gamma(g_s)$, and thus $B_t \leq C |\int_0^t (1+J_\gamma(g_s))
ds|^2$. We finally get $E[\sup_{[0,t]} |V_s^k|^2] \leq 
C_T( 1+E[|V_0|^2]+\int_0^t m_2(g_s) ds + \int_0^t E[|V^k_s|^2]ds
+ |\int_0^t (1+J_\gamma(g_s))
ds|^2)$,
and the conclusion follows by the Gronwall Lemma, since
$\int_0^T m_2(g_s) ds + \int_0^T J_\gamma(g_s)ds < \infty$ by assumption.

\vip

{\it Case $\gamma \in (-3,-2]$.} Since $\gamma<\gamma+1\leq 0$, we obtain
as previously that $B_t \leq C |\int_0^t (1+J_\gamma(g_s))
ds|^2$. A similar argument, using that $\gamma<\gamma+2\leq 0$
(and thus $x^{\gamma+2}\leq 1 + x^\gamma$), yields
$A_t \leq C \int_0^t (1+J_\gamma(g_s)) ds$. We finally find
$E[\sup_{[0,t]} |V_s^k|^2] \leq C_T( E[|V_0|^2]
+ \int_0^t (1+J_\gamma(g_s)) ds + |\int_0^t (1+ J_\gamma(g_s)) ds|^2)$,
and the conclusion follows, since
$\int_0^T J_\gamma(g_s)ds < \infty$ by assumption.
\end{proof}

Tanaka \cite{tanaka} (see also \cite[Lemma 4.7]{fm}) observed the following
elementary fact.

\begin{lem}\label{plp}
Consider a $(\cF_t)_{t\in[0,T]}$-Poisson measure $\mu(ds,dx,d\varphi)$ on 
$[0,T]\times F \times [0,2\pi)$ with intensity measure $ds \nu(dx) d\varphi$,
for some measurable space space $F$ endowed with a nonnegative measure
$\nu$. Then for any predictable map $\varphi_*:\Omega\times [0,T]\times F
\mapsto [0,2\pi)$,  the random measure $\mu_*(ds,dx,d\varphi)$ on 
$[0,T]\times F \times [0,2\pi)$ defined by
\begin{equation}
\mu_*(A)=\int_0^T \int_F \int_0^{2\pi} \indiq_A(s,x,\varphi+\varphi_*(s,x)) 
\mu(ds,dx,d\varphi) \quad \forall\; A\subset [0,T]\times F\times [0,2\pi)
\end{equation}
is again a $(\cF_t)_{t\in[0,T]}$-Poisson measure
with intensity measure $ds \nu(dx) d\varphi$. Of course, we
write $\varphi+\varphi_*(s,x)$ for its value modulo $2\pi$. 
\end{lem}

Our main result will be based on the following
proposition.

\begin{prop}\label{ineqk}
Assume {\bf (A1-A2-A3-A4$(\gamma)$)}, for some $\gamma\in (-3,0]$.
Let $k\geq 1$, $V_0,\tV_0$ two $\cF_0$-measurable $\rd$-valued random 
variables. We also consider $g$ and $\tg$ in 
$\lipdeux \cap L^1([0,T], \cJ_\gamma)$.
Let us finally consider, for each $s\in[0,T]$, 
$R_s \in \cH(g_s,\tg_s)$ such that $s \mapsto R_s$ is measurable.
We may find a $g$-Poisson measure $N$ and a $\tg$-Poisson measure $M$ 
such that, for $V^k$ the $(V_0,g,k,N)$-process and 
$\tV^k$ the $(\tV_0,\tg,k,M)$-process, the following property holds.

(i) If $\gamma \in (-3,0)$, set $\alpha(\gamma)= \min(1/|\gamma|,|\gamma|/2) >0$.
For all $L\geq 1$, all $t\in[0,T]$,
\begin{eqnarray}\label{ccc}
&&E[|V^k_t-\tV^k_t|^2] \leq E[|V_0-\tV_0|^2] + K_T(g,\tg,V_0,\tV_0)
L^{-\alpha(\gamma)} \ala 
&&\hskip1cm+C \intot ds E\left[ 
\intrdd R_s(dv,d\tv) 
\left\{|V^k_s-\tV^k_s|^2 + |v-\tv|^2\} \min(|V^k_s-v|^\gamma 
+ |\tV^k_s-\tv|^\gamma, L) \right\}\right],
\end{eqnarray}
where $K_T(V_0,\tV_0,g,\tg)$ depends only on 
$T,\gamma,\kappa_1,\kappa_2,\kappa_3,g,\tg$
and  $E[|V_0|^2], E[|\tV_0|^2]$.

(ii) If $\gamma=0$, for all $t\in[0,T]$, 
\begin{eqnarray}\label{ccc2}
&E[|V^k_t-\tV^k_t|^2] \leq E[|V_0-\tV_0|^2] +
C \intot ds E\left[ 
\intrdd R_s(dv,d\tv) 
\left\{|V^k_s-\tV^k_s|^2 + |v-\tv|^2\}
\right\}\right].
\end{eqnarray}
\end{prop}

\begin{proof}
Let thus $k\geq 1$, $g,\tg$, $V_0,\tV_0$, and $(R_s)_{s\in [0,T]}$ 
be as in the statement.
We introduce a Poisson measure $\Delta$ on $[0,T]\times(\rd\times\rd)\times
[0,\infty)\times[0,2\pi)$ with intensity measure
$ds R_s(dv,d\tv) dz d\varphi$.
Then, since the restriction of this measure to $z\in[0,k]$ is a.s. finite,
there exists a unique pair of processes $(V^k_t)_{t\in[0,T]}$
and  $(\tV^k_t)_{t\in[0,T]}$, solution of
\begin{eqnarray} 
&&V^k_t= V_0 + \intot \int_{\rd\times\rd} \int_0^k\int_0^{2\pi} 
c(V^k_\sm,v,z,\varphi) \Delta( ds,d(v,\tv),dz,d\varphi), \ala
&&\tV^k_t= \tV_0 +  \intot \int_{\rd\times\rd} \int_0^k\int_0^{2\pi} 
c(\tV^k_\sm,\tv,z,\varphi+\varphi_0(V^k_\sm-v,\tV^k_\sm-\tv)) 
\Delta( ds,d(v,\tv),dz,d\varphi),
\end{eqnarray}
where $\varphi_0$ was introduced in Lemma \ref{tanana}.
Consider now the random measures $N$ and $M$ defined on 
$[0,T]\times\rd\times [0,\infty)\times[0,2\pi)$ by
\begin{eqnarray}
&&N(A)= \int_0^T \int_{\rd\times\rd} \int_0^\infty\int_0^{2\pi} 
\indiq_A (s,v,z,\varphi) \Delta( ds,d(v,\tv),dz,d\varphi), \ala
&&M(A)= \int_0^T \int_{\rd\times\rd} \int_0^\infty\int_0^{2\pi} 
\indiq_A (s,\tv,z,\varphi+\varphi_0(V^k_\sm-v,\tV^k_\sm-\tv)) 
\Delta( ds,d(v,\tv),dz,d\varphi).
\end{eqnarray}
Then $N$ is classically a $g$-Poisson measure, since for each $s$,
$g_s$ is the first marginal of $R_s$. Furthemore, $M$ is a 
$\tg$-Poisson-measure, since for each $s$,
$\tg_s$ is the second marginal of $R_s$, and due Lemma (\ref{plp}).

\vip

Thus $(V^k_t)_{t\in[0,T]}$ (resp. $(\tV^k)_{t\in[0,T]}$) is 
the $(V_0,g,k,N)$-process (resp. the 
$(\tV_0,\tg,k,M)$-process). Next, setting for simplicity
$c:= c(V^k_\sm,v,z,\varphi)$ and 
$\tc:=c(\tV^k_\sm,\tv,z,\varphi+\varphi_0(V^k_\sm-v,\tV^k_\sm-\tv))$,
we get
\begin{eqnarray}
|V^k_t-\tV^k_t|^2=|V_0-\tV_0|^2+ \intot \int_{\rd\times\rd}\int_0^k \intzdp
\left\{ |V^k_\sm+c - \tV^k_\sm - \tc |^2 -|V^k_\sm - \tV^k_\sm |^2
\right\} \ala 
\Delta(ds,d(v,\tv),dz,d\varphi).
\end{eqnarray}
Hence, taking expectations,
\begin{eqnarray}\label{toc0}
&&E[|V^k_t-\tV^k_t|^2]=E[|V_0-\tV_0|^2]\ala
&&+\int_0^t ds E \left\{ \int_{\rd\times\rd}R_s(dv,d\tv) \int_0^k dz
\int_0^{2\pi}d\varphi
[|c-\tc|^2 + 2 (V^k_s-\tV^k_s).(c-\tc)] \right\} .
\end{eqnarray}
Now, using (\ref{esti2}), (\ref{esti4}) and {\bf (A4)$(\gamma)$}, 
we easily deduce that a.s.,
\begin{eqnarray}\label{toc1}
&&\int_0^k dz
\int_0^{2\pi}d\varphi
[|c-\tc|^2 + 2 (V^k_s-\tV^k_s).(c-\tc)] \ala
&& \hskip2cm \leq C 
(|V^k_s-\tV^k_s|^2 + |v-\tv|^2)(|V^k_s-v|^\gamma+ |\tV^k_s-\tv|^\gamma),
\end{eqnarray}
while using (\ref{esti1}), (\ref{esti3}) and {\bf (A4)$(\gamma)$},  
we obtain a.s.
\begin{eqnarray}\label{toc2}
&&\int_0^k dz
\int_0^{2\pi}d\varphi
[|c-\tc|^2 + 2 (V^k_s-\tV^k_s).(c-\tc)]\ala
&&\leq C
|V^k_s-v|^{2+\gamma} + C |\tV^k_s-\tv|^{2+\gamma} 
+ C |V^k_s-\tV^k_s| \left\{|V^k_s-v|^{1+\gamma}+ |\tV^k_s-\tv|^{1+\gamma} 
\right\}.
\end{eqnarray}

If $\gamma=0$, (\ref{ccc2}) 
follows immediately from (\ref{toc0}) and (\ref{toc1}).
We thus now assume that $\gamma\in (-3,0)$.
Let $L\geq 1$ be fixed. We insert (\ref{toc1}) (resp. (\ref{toc2})) in (\ref{toc0})
when $|V^k_s-v|^\gamma+|\tV^k_s-\tv|^\gamma \leq L$ 
(resp. $|V^k_s-v|^\gamma+|\tV^k_s-\tv|^\gamma > L$),
and we obtain
\begin{eqnarray}\label{depart}
&&E[|V^k_t-\tV^k_t|^2] \leq E[|V_0-\tV_0|^2] + 
C\sum_{i=1}^4(I_t^{i,L}+\tI^{i,L}_t) 
\ala
&&\hskip1cm+C \intot ds E\left[ 
\intrdd R_s(dv,d\tv) 
\left\{|V^k_s-\tV^k_s|^2 + |v-\tv|^2\}\min(|V^k_s-v|^\gamma 
+ |\tV^k_s-\tv|^\gamma,L)\right\} \right],
\end{eqnarray}
where
\begin{eqnarray}
I^{1,L}_t &:=&  \intot ds E \left[ \intrdd R_s(dv,d\tv) |V^k_s-v|^{2+\gamma}
\indiq_{\{|V^k_s-v|^\gamma>L/2\}} \right],\ala
I^{2,L}_t &:=& \intot ds E \left[ \intrdd R_s(dv,d\tv) |\tV^k_s-\tv|^{2+\gamma}
\indiq_{\{|V^k_s-v|^\gamma>L/2\}} \right], \ala
I^{3,L}_t   &:=&  \intot ds E \left[ \intrdd R_s(dv,d\tv) |V^k_s-\tV^k_s| 
|V^k_s-v|^{1+\gamma} \indiq_{\{|V^k_s-v|^\gamma>L/2\}}\right],\ala
I^{4,L}_t   &:=&  \intot ds E \left[ \intrdd R_s(dv,d\tv) |V^k_s-\tV^k_s| 
|\tV^k_s-\tv|^{1+\gamma} \indiq_{\{|V^k_s-v|^\gamma>L/2\}}\right],
\end{eqnarray}
and where $\tI^{1,L}_t,...,\tI^{4,L}_t$ have the same expressions replacing 
$V^k_s,\tV^k_s,v,\tv$ by $\tV^k_s,V^k_s,\tv,v$. 

We first treat the case of $I^{1,L}$. Since $R_s \in \cH(g_s,\tg_s)$ and $\gamma\in(-3,0)$, 
and using notation (\ref{cjalpha}) one has
\begin{eqnarray}
I^{1,L}_t &=&  \intot ds E \left[ \intrd g_s(dv) |V^k_s-v|^{2+\gamma}
\indiq_{\{|V^k_s-v|^\gamma>L/2\}} \right]\ala
&\leq & C L^{2/\gamma} \intot ds E \left[ \intrd g_s(dv) 
|V^k_s-v|^{\gamma}\right]
\leq  C L^{2/\gamma} \intot ds J_\gamma(g_s) \leq K_T(g) L^{-\alpha(\gamma)},
\end{eqnarray}
since $2/|\gamma|\geq \alpha(\gamma)$. Similarly, since $1/|\gamma|\geq \alpha(\gamma)$
\begin{eqnarray}
I^{3,L}_t &=&  \intot ds E \left[|V^k_s-\tV^k_s| \intrd g_s(dv) 
|V^k_s-v|^{1+\gamma}
\indiq_{\{|V^k_s-v|^\gamma>L/2\}} \right]\ala
&\leq & C L^{1/\gamma} E\left[\sup_{[0,T]} (|V^k_s|+|\tV^k_s|) \right] 
\intot ds J_\gamma(g_s) \leq K_T(g,\tg,V_0,\tV_0) L^{-\alpha(\gamma)},
\end{eqnarray}
where we used Lemma \ref{m2k}.
We now study $I^{2,L}_t$ when $\gamma\in [-2,0)$, so that $\gamma+2\in[0,2)$.
Using the H\"older inequality, we get
\begin{eqnarray}
I^{2,L}_t &\leq&  E\left[
\left(\intot ds \intrdd R_s(dv,d\tv) |\tV^k_s-\tv|^2
\right)^{\frac{2+\gamma}{2}} 
\left(\intot ds \intrdd R_s(dv,d\tv) 
\indiq_{\{|V^k_s-v|^\gamma>L/2\}}\right)^{\frac{|\gamma|}{2}}\right] \ala
&=&E\left[
\left(\intot ds \intrd \tg_s(d\tv) |\tV^k_s-\tv|^2\right)^{\frac{2+\gamma}{2}} 
\left(\intot ds \intrd g_s(dv) 
\indiq_{\{|V^k_s-v|^\gamma>L/2\}}\right)^{\frac{|\gamma|}{2}}\right] \ala
&\leq & C E\left[
\left(\intot ds (m_2(\tg_s) + |\tV^k_s|^2)\right)^{\frac{2+\gamma}{2}} 
\left(\intot ds \intrd g_s(dv) \frac{|V^k_s-v|^\gamma}{L/2} 
\right)^{\frac{|\gamma|}{2}}
\right]\ala
&\leq & C L^{\gamma/2}\{1 + \sup_{[0,T]} m_2(\tg_s)+ E[ \sup_{[0,T]} |\tV^k_s|^2]\}
\left(\intot ds J_\gamma(g_s)ds \right)^{\frac{|\gamma|}{2}}
\leq K_T(g,\tg,\tV_0) L^{-\alpha(\gamma)},
\end{eqnarray}
by Lemma \ref{m2k}, since $\frac{\gamma+2}{\gamma}\in [0,1]$ and 
$|\gamma|/2 \geq \alpha(\gamma)$. We next study $I^{2,L}_t$ when 
$\gamma \in (-3,-2)$, so that $\gamma+2\in (-1,0)$. The H\"older inequality
yields
\begin{eqnarray}
I^{2,L}_t &\leq&  E\left[
\left(\intot ds \intrdd R_s(dv,d\tv)
|\tV^k_s-\tv|^\gamma\right)^{\frac{2+\gamma}{\gamma}} 
\left(\intot ds \intrdd R_s(dv,d\tv) 
\indiq_{\{|V^k_s-v|^\gamma>L/2\}}\right)^{\frac{2}{|\gamma|}}\right] \ala
&=&E\left[
\left(\intot ds \intrd \tg_s(d\tv) |\tV^k_s-\tv|^\gamma
\right)^{\frac{2+\gamma}{\gamma}} 
\left(\intot ds \intrd g_s(dv) 
\indiq_{\{|V^k_s-v|^\gamma>L/2\}}\right)^{\frac{2}{|\gamma|}}\right] \ala
&\leq & C \left(\intot ds J_\gamma(\tg_s) \right)^{\frac{2+\gamma}{\gamma}} 
\left(\intot ds J_\gamma(g_s) \right)^{\frac{2}{|\gamma|}} L^{2/\gamma}
\leq K_T(g,\tg) L^{-\alpha(\gamma)},
\end{eqnarray}
since $2/|\gamma|\geq \alpha(\gamma)$. 
Let us now upperbound
$I^{4,L}_t$ in the case $\gamma\in [-1,0)$, so that $1+\gamma \in [0,1)$.
Using the H\"older inequality, one finds as usual (since 
$(1+\gamma)/2\leq 1/2$),
\begin{eqnarray}
I^{4,L}_t &\leq&  E\Big[ \sup_{[0,T]}|V^k_s-\tV^k_s| \times
\left(\intot ds \intrd \tg_s(d\tv) |\tV^k_s-\tv|^2 
\right)^{\frac{1+\gamma}{2}}\ala
&&\hskip4cm \times \left(\intot ds \intrd g_s(dv)
\indiq_{\{|V^k_s-v|^\gamma>L/2\}}\right)^{\frac{1-\gamma}{2}}\Big] \ala
& \leq & C E \left[ \left\{ \sup_{[0,T]}(|V^k_s|+|\tV^k_s|) \right\}
\left\{1+ \intot ds (m_2(g_s)+|V^k_s|^2)\right\}^{\frac{1}{2}}
\left\{\frac{1}{L}\intot ds J_\gamma(g_s) \right\}^{\frac{1-\gamma}{2}}
\right]\ala
& \leq & C L^{\frac{\gamma-1}{2}} \{1+ \sup_{[0,T]}m_2(\tg_s) + 
E[\sup_{[0,T]}(|V^k_s|^2+|\tV^k_s|^2)]\}
\left( \intot ds J_\gamma(g_s) \right)^{\frac{1-\gamma}{2}}\ala
&\leq& K_T(g,\tg,V_0, \tV_0) L^{-\alpha(\gamma)},
\end{eqnarray}
since $(1+|\gamma|)/2 \geq \alpha(\gamma)$ and by Lemma \ref{m2k}. 
Finally we consider
$I^{4,L}_t$ in the case $\gamma\in (-3,-2)$, so that $1+\gamma \in 
(\gamma,0)$: 
\begin{eqnarray}
I^{4,L}_t &\leq& 
E\Big[ \sup_{[0,T]}|V^k_s-\tV^k_s| \times
\left(\intot ds \intrd \tg_s(d\tv) |\tV^k_s-\tv|^\gamma
\right)^{\frac{1+\gamma}{\gamma}}\ala
&&\hskip4cm \times \left(\intot ds \intrd g_s(dv)
\indiq_{\{|V^k_s-v|^\gamma>L/2\}}\right)^{\frac{1}{|\gamma|}}\Big] \ala
& \leq & C L^{1/\gamma} E[\sup_{[0,T]}(|V^k_s|+|\tV^k_s|)]
\left( \intot ds J_\gamma(\tg_s) \right)^{\frac{1+\gamma}{\gamma}}
\left( \intot ds J_\gamma(g_s) \right)^{\frac{1}{|\gamma|}}\ala
&\leq& K_T(g,\tg,V_0,\tV_0) L^{-\alpha(\gamma)}.
\end{eqnarray}
since $1/|\gamma| \geq \alpha(\gamma)$. Using the same computations for $\tI^{1,L}_t,...\tI^{4,L}_t$, 
(\ref{ccc}) follows immediately.
\end{proof}

Admitting for a moment Lemma \ref{cacv}, we give the

\begin{proof}[Proof of Theorem \ref{maintheo}]
We thus assume {\bf (A1-A2-A3-A4$(\gamma)$)} for some
$\gamma \in (-3,0)$, the case $\gamma=0$ is easier and left to the reader.
We consider two weak solutions $(f_t)_{t\in [0,T]}$ and 
$(\tf_t)_{t\in [0,T]}$ to (\ref{be}) lying in 
$\lipdeux \cap L^1([0,T]),\cJ_\gamma)$.

We consider two $\cF_0$-measurable random variables 
$V_0\sim f_0$ and $\tV_0 \sim \tf_0$, such that 
$W_2^2(f_0,\tf_0)= E[|V_0-\tV_0|^2]$, and for each $s\in [0,T]$,
we consider $R_s\in \cH(f_s,\tf_s)$ such that 
$W_2^2(f_s,\tf_s)= \int_{\rd\times\rd} |v-\tv|^2 R_s(dv,d\tv)$.
Due to Fontbona-Gu\'erin-M\'el\'eard \cite[Theorem 1.3]{fgm}, $s\mapsto R_s$ 
is measurable: indeed, since by assumption $f_t$ has a density for all
$t\in[0,T]$, the minimizer $R_t$ is unique (see e.g. Villani 
\cite[Theorem 2.12]{villani2}).

Finally, for each $k\geq 1$, we consider some $(V_0,f,k,N)$-process
$(V^k_t)_{t\in [0,T]}$
and some $(\tV_0,\tf,k,M)$ process $(\tV^k_t)_{t\in [0,T]}$, 
coupled as in Proposition \ref{ineqk}.

We set $w^k_t:=E[|V^k_t-\tV^k_t|^2]$ 
for each $k\geq 1$, each $t\in [0,T]$.
Using Lemma \ref{cacv}, we deduce that for all $t\in [0,T]$,
\begin{equation}
u_t:=W_2^2(f_t,\tf_t) \leq \limsup_k w^k_t =:w_t.
\end{equation}
We observe at once that due to Lemma \ref{m2k} and by assumption on $f,\tf$,
\begin{equation}\label{toutborne}
\sup_k \sup_{[0,T]} w^k_t + \sup_{[0,T]} w_t <\infty.
\end{equation}
Due to Proposition \ref{ineqk}, we know that for all $L\geq 1$,
all $k\geq 1$, all $t\in [0,T]$,
\begin{eqnarray}
w^k_t & \leq & u_0 + K(T,f,\tf,V_0,\tV_0) L^{-\alpha(\gamma)} \ala
&& + C\intot ds E \left[ |V^k_s-\tV^k_s|^2 \intrdd R_s(dv,d\tv)
(|V^k_s-v|^\gamma+|\tV^k_s-\tv|^\gamma ) \right] \ala
&& + C\intot ds E \left[ \intrdd R_s(dv,d\tv) |v-\tv|^2
\min(|V^k_s-v|^\gamma,L)  \right] \ala
&&+ C\intot ds E \left[ \intrdd R_s(dv,d\tv) |v-\tv|^2
\min(|\tV^k_s-\tv|^\gamma,L)  \right]\ala
&=:& u_0 + K(T,f,\tf,V_0,\tV_0) L^{-\alpha(\gamma)} + CA_k(t,L) + C B_k(t,L)
+ C\tB_k(t,L).
\end{eqnarray}
First, recalling (\ref{cjalpha}) and that $R_s \in \cH(f_s,\tf_s)$,
we observe that
\begin{eqnarray}
A_k(t,L) &=&  \int_0^t ds    E \left[ |V^k_s-\tV^k_s|^2 \intrd (f_s(dv) 
|V^k_s-v|^\gamma+ \tf_s(d\tv) |\tV^k_s-\tv|^\gamma ) \right]  \ala
&\leq &  \int_0^t ds \; w^k_s J_\gamma(f_s+\tf_s).
\end{eqnarray}
Hence for all $L\geq 1$,
using (\ref{toutborne}), the Lebesgue Theorem and that $J_\gamma(f_s+\tf_s)$
belongs to $L^1([0,T])$ by assumption, 
\begin{equation}\label{dejabien}
\limsup_k A_k(t,L) \leq \intot ds \; w_s J_\gamma(f_s+\tf_s).
\end{equation}
Next, one easily checks that 
for each $s\in [0,T]$, the map $v_* \mapsto \alpha_s(v_*):=
\intrdd R_s(dv,d\tv)|v-\tv|^2 \min(|v_*-v|^\gamma,L)$ is continuous
on $\rd$ and bounded by $2 L (m_2(f_s)+m_2(\tf_s))$.
Since furthermore $\int_0^T (m_2(f_s)+m_2(\tf_s)) ds <\infty$ by assumption,
we easily deduce from Lemma \ref{cacv} and the Lebesgue Theorem that
\begin{eqnarray}
\lim_k B_k(t,L) & = & \intot ds 
\intrdd R_s(dv,d\tv)|v-\tv|^2 \intrd f_s(dv_*) \min(|v_*-v|^\gamma,L)  \ala
&\leq & \intot ds 
\intrdd R_s(dv,d\tv)|v-\tv|^2 \intrd f_s(dv_*) |v_*-v|^\gamma  \ala
&\leq &  \intot ds \; u_s J_\gamma(f_s) \leq  \intot ds \; w_s J_\gamma(f_s).
\end{eqnarray}
Using the same computation for $\tB_k(t,L)$, we finally obtain,
for all $L\geq 1$, 
\begin{equation}
w_t \leq u_0 + K(T,f,\tf,V_0,\tV_0) L^{-\alpha(\gamma)} + C\intot
w_s J_\gamma(f_s+\tf_s) ds.
\end{equation}
Making $L$ tend to infinity, and using the generalized Gronwall Lemma 
\ref{ggl},
we deduce that for $t\in [0,T]$, 
$w_t \leq u_0 \exp(C \int_0^t J_\gamma(f_s+\tf_s)ds )$.
Since $W^2_2(f_t,\tf_t)=u_t\leq w_t$, this concludes the proof.
\end{proof}

It remains to prove the convergence Lemma \ref{cacv}. To this aim, we first
prove a uniqueness result for a linearized Boltzmann equation.

\begin{lem}\label{unili}
Assume {\bf (A1-A2-A3-A4$(\gamma)$)}, for some $\gamma \in (-3,0]$.
Consider a weak solution $f=(f_t)_{t\in[0,T]} \in \lipdeux\cap 
L^1([0,T],\cJ_\gamma)$ 
to (\ref{be}). Assume that for some 
$g=(g_t)_{t\in[0,T]} \in \lipdeux$ for all $\phi\in C^2_c$, all $t\in [0,T]$,
\begin{eqnarray}\label{lbe}
\intrd \phi(v)g_t(dv)&=&\intrd \phi(v)f_0(dv) + \intot ds \intrd 
g_s(dv) \intrd f_s(dv_*)
\ctA\phi(v,v_*),
\end{eqnarray}
with $\ctA\phi$ defined by (\ref{agood}).
Then $g=f$.
\end{lem}

\begin{proof}
We thus assume {\bf (A1-A2-A3-A4$(\gamma)$)} for some $\gamma\in (-3,0]$,
and (unfortunately) use some martingale
problems techniques. We consider a weak solution $f=(f_t)_{t\in[0,T]} 
\in \lipdeux\cap L^1([0,T],\cJ_\gamma)$  to (\ref{be}).
We also consider, for each $t\geq 0$ the operator $\ctA_t$ 
defined for $\phi \in C^2_\infty$ and $v\in\rd$ by
\begin{equation}\label{ctat}
\ctA_t\phi(v)=\intrd f_t(dv_*) \ctA\phi(v,v_*).
\end{equation}
We will prove that for any $\mu \in \pdeux$, there exists at most one
$g\in \lipdeux$ such that
for all $t\geq 0$, all $\phi\in C^2_c$,
\begin{equation}\label{eqlir}
\intrd \phi(v)g_t(dv) = \intrd \phi(v)\mu(dv) 
+ \intot ds \intrd g_s(dv) \ctA_s \phi(v).
\end{equation}
Since by assumption, $f$ and $g$ solve this equation with $\mu=f_0$,
this will conclude the proof.

\vip

{\it Step 1.} Let $\mu \in \pdeux$. 
A c\`adl\`ag adapted $\rd$-valued stochastic process $(V_t)_{t\in[0,T]}$ on
some filtered probability space $(\Omega,\cF,(\cF_t)_{t\in[0,T]}, P)$
is said to solve the martingale problem $MP((\ctA_t)_{t\in[0,T]}, \mu)$
if $P\circ V_0^{-1}=\mu$ and if for all $\phi \in C^2_c$, 
$(M^\phi_t)_{t\in[0,T]}$ is a $(\Omega,\cF,(\cF_t)_{t\in[0,T]}, P)$-martingale,
where
\begin{equation}\label{mpmp}
M^\phi_t = \phi(V_t) - \intot \ctA_s\phi(V_s) ds.
\end{equation}
Assume for a moment that:

(i) there exists a countable subset $(\phi_k)_{k\geq 1}\subset C^2_c$
such that for all $t\in[0,T]$, the closure 
(for the bounded pointwise convergence)
of $\{(\phi_k,\ctA_t\phi_k), k\geq 1\}$
contains $\{(\phi, \ctA_t \phi), \phi\in C^2_c\}$,

(ii) for each $v_0\in\rd$, there exists a solution to
$MP((\ctA_t)_{t\in[0,T]}, \delta_{v_0})$,

(iii) for each $v_0\in\rd$, uniqueness (in law) holds for 
$MP((\ctA_t)_{t\in[0,T]}, \delta_{v_0})$.

Then, due to Bhatt-Karandikar \cite[Theorem 5.2]{bk} (see also
Remark 3.1 and Theorem 5.1 in \cite{bk} and Theorem B.1 in \cite{kh}),
uniqueness for (\ref{eqlir}) holds. 

\vip

{\it Step 2.} First, (i) holds: consider any countable subset 
$(\phi_k)_{k\geq 1}\subset
C^2_c$ dense in $C^2_c$, in the sense that for $\psi \in C^2_c$ with supp 
$\psi \subset \{|v|\leq x\}$,
there exists a subsequence $\phi_{k_n}$ such that supp $\phi_{k_n}\subset  
\{|v|\leq 2x\}$,
and $\lim_{n \to \infty} (||\psi-\phi_{k_n} ||_\infty +||\psi'-\phi'_{k_n} 
||_\infty 
+||\psi''-\phi''_{k_n} ||_\infty) =0$.

We then have to prove that, for $t\in [0,T]$,

(a) $\ctA_t \phi_{k_n}(v)$ tends to $\ctA_t \psi(v)$ for all $v\in \rd$,

(b) and that $\sup_n ||\ctA_t \phi_{k_n}||_\infty <\infty$.

First, using Lemma \ref{fundest}-(ii) and {\bf (A4)}$(\gamma)$, we get, 
for $v\in \rd$,
\begin{eqnarray}
&&|\ctA_t \phi_{k_n}(v) - \ctA_t \psi(v)  |  =  
|\ctA_t [\phi_{k_n}-\psi](v) | \ala
&&\leq C (||\psi'-\phi'_{k_n} ||_\infty 
+||\psi''-\phi''_{k_n} ||_\infty)\intrd f_t(dv_*)  (|v-v_*|+|v-v_*|^2) 
|v-v_*|^\gamma,
\end{eqnarray}
which tends to $0$ as $n$ tends to infinity, 
provided $\alpha_t:=\intrd f_t(dv_*)  (|v-v_*|+|v-v_*|^2) 
|v-v_*|^\gamma<\infty$.
But $\alpha_t \leq \intrd f_t(dv_*)  (1+2|v-v_*|^2) |v-v_*|^\gamma \leq  
J_\gamma(f_t) + 
2\intrd f_t(dv_*)|v-v_*|^{2+\gamma}$. If $\gamma \in [-2,0]$, then 
$|v-v_*|^{2+\gamma} \leq 1+2|v|^2+2|v_*|^2$, so that $\alpha_t \leq 
J_\gamma(f_t)+C(m_2(f_t)+
1+|v|^2)<\infty$ by assumption. If $\gamma \in (-3,-2)$, then 
$|v-v_*|^{2+\gamma} \leq 1+|v-v_*|^\gamma$, so that $\alpha_t \leq 3 
J_\gamma(f_t) +1<\infty$.
Thus (a) holds, and it remains to prove (b). Set $M:= \sup_n 
(||\phi_{k_n}||_\infty 
+||\phi'_{k_n} ||_\infty +||\phi''_{k_n} ||_\infty)$. If $\gamma \in (-3,-2]$,
then one easily deduces from {\bf (A4)}$(\gamma)$ and Lemma 
\ref{fundest}-(ii) that 
for all $t$, all $n$, all $v$, $|\ctA_t \phi_{k_n} (v)| \leq C M 
(1+J_\gamma(f_t))$, which implies (b).
If $\gamma\in (-2,0]$, we use Lemma \ref{fundest}-(iii), and get, since 
$|v-v_*|^{2+\gamma} \leq 1+2|v|^2+2|v_*|^2$,
\begin{eqnarray}
|\ctA_t \phi_{k_n}(v)|&\leq& C M \intrd f_t(dv_*) \left[ 1+ 
|v-v_*|^{2+\gamma}\indiq_{\{|v|\leq 4x \}}
+ \frac{|v-v_*|^{2+\gamma}}{|v|^2}  \indiq_{\{|v|\geq 4x \}}  \right] \ala
&\leq & C M \intrd f_t(dv_*) \left[ 1+ x^2 + |v_*|^2 + x^{-2} 
(1+|v_*|^2)\right] \ala
&\leq& C M ( 1+ x^2 + x^{-2} + m_2(f_t)(1+x^{-2})).
\end{eqnarray}

\vip

{\it Step 3.} Classical arguments (see e.g.
Tanaka \cite[Section 4]{tanaka76} or
Desvillettes-Graham-M\'el\'eard \cite[Theorem 3.8]{dgm})
yield that a process $(V_t)_{t\in[0,T]}$ 
on some
filtered  probability space $(\Omega,\cF,(\cF_t)_{t\in[0,T]}, P)$
is
a solution to $MP((\ctA_t)_{t\in[0,T]}, \delta_{v_0})$ 
if and only if 
there exists, on a possibly enlarged probability space,  
a $(\cF_t)_{t\in[0,T]}$-adapted $f$-Poisson measure 
$N(dt,dv,dz,d\varphi)$
such that (recall the expression (\ref{agood}) 
of $\ctA$, and that $\tN$ stands for the compensated Poisson measure)
\begin{eqnarray}\label{sde}
V_t=v_0+\intot\intrd\intzi\intzdp c(V_\sm,v,s,\varphi)
\tN(ds,dv,dz,d\varphi)\ala
-\kappa_0 \intot ds \intrd f_s(dv) \Phi(|V_s-v|) (V_s-v).
\end{eqnarray}
We thus just have to prove the existence and uniqueness in law 
for solutions to (\ref{sde}). 

\vip

{\it Step 4.} We now check that for $(V_t)_{t\in [0,T]}$ 
a solution to (\ref{sde}), 
\begin{equation}\label{unik}
E\left[\sup_{[0,T]} |V_t|^2 \right] <\infty.
\end{equation}
We introduce, for $n\geq 1$, the stopping time $\tau_n=\inf \{t\geq 0, 
|V_t|\geq n\}$.
Using the Doob and Cauchy-Schwarz inequalities, thanks to (\ref{esti1}) and {\bf (A4)}($\gamma$) we get
\begin{eqnarray}
E\left[\sup_{[0,t\land \tau_n]} |V_s|^2 \right] \leq  C|v_0|^2 
+ C E\left[\int_0^{t\land \tau_n} ds \intrd f_s(dv) \int_0^\infty dz 
\intzdp d\varphi 
|c(V_\sm,v,s,\varphi) |^2  \right] \ala
+ C_T E\left[\left| \int_0^{t\land \tau_n} ds \intrd f_s(dv) 
\Phi(|V_\sm-v|)|V_\sm-v| \right|^2
\right] \\
\leq  C|v_0|^2  + C E\left[\int_0^{t\land \tau_n} ds \intrd f_s(dv) 
|V_s-v|^{2+\gamma} \right]
+  C_TE\left[\left|\int_0^{t\land \tau_n} ds \intrd f_s(dv) 
|V_s-v|^{1+\gamma} \right|^2 \right].\nonumber
\end{eqnarray}
Separating as usual the cases $\gamma \in (-3,-2]$, $\gamma \in (-2,-1]$ 
and $\gamma \in (-1,0]$
(see e.g. the proof of Lemma \ref{m2k}), we obtain in any case
\begin{eqnarray}
E\left[\sup_{[0,t\land \tau_n]} |V_s|^2 \right] \leq  C|v_0|^2 + 
C_T E\left[\int_0^{t\land \tau_n} ds |V_s|^2 \right] + C_T \int_0^T 
J_\gamma(f_s)ds\ala 
+C_T \left( \int_0^T J_\gamma(f_s) ds\right)^2  + C_T \int_0^T ds m_2(f_s) \ala
\leq C_T(v_0,f) + C_T\intot ds E\left[ \sup_{[0,s\land \tau_n]} |V_u|^2 
\right].
\end{eqnarray}
The Gronwall Lemma ensures us that for all $n\geq 1$,  
$E\left[\sup_{[0,T\land \tau_n]} |V_s|^2 \right] \leq C_T(v_0,f)e^{TC_T}$.
We immediately deduce that a.s., $\lim_n \tau_n=\infty$, and then that 
(\ref{unik}) holds.

\vip

{\it Step 5.} Let $(V_t)_{t\in [0,T]}$ be a c\`adl\`ag  
adapted solution to (\ref{sde}), for some $f$-Poisson measure $N$.
Recall Lemma \ref{tanana}, 
and define $(\tV^k_t)_{t\in [0,T]}$ as the solution (which clearly exists
and is unique
since $\indiq_{\{z\leq k\}}N(ds,dv,dz,d\varphi)$ is a.s. finite) 
\begin{eqnarray}
&&\tV^k_t=v_0 + \intot \intrd\int_0^k\intzdp  c(\tV^k_\sm,v,z,\varphi+
\varphi_0(V_\sm-v,\tV^k_\sm-v)) N(ds,dv,dz,d\varphi).
\end{eqnarray}
The map $\varphi_0(V_\sm-v_,\tV^k_\sm-v)$ being predictable, we deduce
from Lemma \ref{plp} that $N_0$ defined by $N_0(A)=\int_0^T 
\int_\rd\int_0^\infty \int_0^{2\pi} 
\indiq_A(s,v,z,\varphi+\varphi_0(V_\sm-v,\tV^k_\sm-v))N(ds,dv,dz,d\varphi)$ 
is still a $f$-Poisson measure. 
Hence  $(\tV^k_t)_{t\in[0,T]}$ is nothing but the $(v_0,f,k,N_0)$-process, 
and its law is entirely
determined by $v_0$ and $f$, see Notation \ref{notagpro}.

\vip

We will now show that $(\tV^k_t)_{t\geq 0}$ goes
in probability to $(V_t)_{t\geq 0}$, which will yield the uniqueness
of the law of $(V_t)_{t\geq 0}$ and thus will end the proof of (iii). 
To this end,
we first observe that due to Step 4 and Lemma \ref{m2k},
\begin{eqnarray}\label{bl2}
C(T,f,v_0):= E\left[ \sup_{[0,T]} |V_t|^2\right] + \sup_k E\left[ 
\sup_{[0,T]} |\tV^k_t|^2\right]
<\infty.
\end{eqnarray}
Then, we may rewrite, recalling (\ref{dfhzero})
\begin{eqnarray}
\tV^k_t&=&v_0+ \intot \intrd\int_0^k\intzdp  c(\tV^k_\sm,v,z,\varphi+
\varphi_0(V_\sm-v,\tV^k_\sm-v)) \tN(ds,dv,dz,d\varphi)\ala
&& + \int_0^t ds \intrd f_s(dv) h_0^k(|\tV^k_s-v|) (\tV^k_s-v).
\end{eqnarray}
This is due to the fact that $\int_0^k dz \int_0^{2\pi} d\varphi 
c(V,v,z,\varphi)
= - (V-v) h_0^k(|V-v|)$. Hence, 
\begin{eqnarray}
(V_t-\tV^k_t) &=& A^{1,k}_t+ A^{2,k}_t + A^{3,k}_t+ A^{4,k}_t,
\end{eqnarray}
where
\begin{eqnarray*}
A^{1,k}_t&:=& \intot \intrd\int_0^k \intzdp [c(V_\sm,v,z,\varphi)-
c(\tV^k_\sm,v,z,\varphi+
\varphi_0(V_\sm-v,\tV^k_\sm-v))]\tN(ds,dv,dz,d\varphi),\ala
A^{2,k}_t&:=& \intot \intrd\int_k^\infty \intzdp c(V_\sm,v,z,\varphi) 
\tN(ds,dv,dz,d\varphi),\ala
A^{3,k}_t&:=& \int_0^t ds \intrd f_s(dv) [h_0^k(|V_s-v|) (V_s-v)
- h_0^k(|\tV^k_s-v|) (\tV^k_s-v) ],\ala
A^{4,k}_t&:=& \int_0^t ds \intrd f_s(dv) [\kappa_0 \Phi(|V_s-v|) (V_s-v) 
- h_0^k(|V_s-v|) (V_s-v) ].
\end{eqnarray*}

First, we immediately deduce from the 
Doob inequality, (\ref{esti2}) and {\bf (A4)}$(\gamma)$, that
\begin{eqnarray}
E\left[\sup_{[0,t]}|A^{1,k}_s|^2\right]&\leq& C \intot ds \intrd f_s(dv) 
E\left[|V_s-\tV^k_s|^2 (|V_s-v|^\gamma +
|\tV^k_s-v|^\gamma )\right] \ala
&\leq & C \intot ds E\left[|V_s-\tV^k_s|^2 \right] 
J_\gamma(f_s).
\end{eqnarray}
Next, Doob's inequality, (\ref{mesti1}) and {\bf (A4)}$(\gamma)$ yield
\begin{equation}
E\left[\sup_{[0,t]}|A^{2,k}_s|^2\right]\leq C \int_0^T ds \intrd 
f_s(dv) E\left[|V_s-v|^{2+\gamma}
\e_0^k(|V_s-v|)\right] \to 0
\end{equation}
as $k$ tends to infinity, since due to Lemma \ref{moinsfundest}, 
$\e_0^k$ is bounded
and tends simply to $0$, and since $|V_s-v|^{2+\gamma}$ belongs to 
$L^1(dsf_s(dv)P(d\omega))$ (as usual, if $2+\gamma \geq 0$, this follows from 
(\ref{bl2}) and the fact that $f\in \lipdeux$, while if $2+\gamma<0$, we just
use that $f\in \lunjgamma$).

Using (\ref{mesti2}), {\bf (A4)}$(\gamma)$, and then the Cauchy-Schwarz 
inequality, we obtain
\begin{eqnarray}
E\left[\sup_{[0,t]}|A^{3,k}_s|^2\right]
&\leq& CE \left[ \left( \intot ds \intrd f_s(dv) 
|V_s-\tV^k_s| (|V_s-v|^\gamma + |\tV^k_s-v|^\gamma )\right)^2\right] \ala
&\leq & C E \left[ \left( \intot ds |V_s-\tV^k_s| J_\gamma(f_s) \right)^2
\right]\ala
&\leq& C \left( \intot ds E[|V_s-\tV^k_s|^2] J_\gamma(f_s) \right)
\left( \int_0^T ds J_\gamma(f_s) \right).
\end{eqnarray}
Finally, (\ref{mesti3}) and {\bf (A4)}$(\gamma)$ allow us to obtain,
\begin{eqnarray}
E\left[\sup_{[0,t]}|A^{4,k}_s|^2\right]
&\leq& CE \left[ \left( \int_0^T ds \intrd f_s(dv) |V_s-v|^{1+\gamma}
\e_0^k(|V_s-v|)\right)^2\right] \to 0
\end{eqnarray}
using similar arguments as for the study of $A^{1,k}$.
We thus obtain, for some $\eta_k$ going to $0$,
some constant $C(T,f)$,
\begin{equation}
E\left[|V_t-\tV^k_t|^2\right] 
\leq \eta_k + C(T,f) \intot ds E[|V_s-\tV^k_s|^2] J_\gamma(f_s).
\end{equation}
The generalized Gronwall Lemma \ref{ggl} and the fact that $f\in\lunjgamma$
by assumption allow us to conclude that
\begin{equation}
E\left[\sup_{[0,T]}|V_t-\tV^k_t|^2\right] \leq \eta_k \exp[C(T,f) \int_0^T 
ds J_\gamma(f_s)]
\to 0
\end{equation}
as $k$ tends to infinity. Hence $(\tV^k_t)_{t\in [0,T]}$ goes in probability to
$(V_t)_{t\in [0,T]}$.

\vip

{\it Step 6.} It remains to prove (ii), i.e. the existence for 
$MP((\ctA_t)_{t\in[0,T]}, \delta_{v_0})$. We use to this
aim a Picard iteration. Let $N$ be a $f$-Poisson measure as in Step 2.
We consider the constant process $V^0_t\equiv v_0$, we set $\varphi^*_0=0$ and
we define recursively
\begin{eqnarray}\label{rec}
V^{n+1}_t &=& v_0 + \intot \intrd \int_0^\infty \intzdp c(V^n_\sm,v,z,\varphi
+\varphi^*_n(s,v))
\tN(ds,dv,dz,d\varphi) \ala
&& - \kappa_0 \intot ds \intrd f_s(dv) \Phi(|V^n_s-v|)(V^n_s-v),
\end{eqnarray}
and $\varphi^*_{n+1}(s,v)=\varphi^*_n(s,v)+
\varphi_0(V^{n+1}_\sm-v,V^n_\sm-v)$, where $\varphi_0$ is defined by Lemma \ref{tanana}.
A computation based on Doob's inequality using (\ref{esti2}) and that
$|x\Phi(x)-y\Phi(y)|\leq \min(x,y)|\Phi(x)-\Phi(y)|+|x-y|(\Phi(x)+\Phi(y))
\leq C |x-y|(x^\gamma+y^\gamma)$
yields that for all $t\in [0,T]$, all $n\geq 1$,
\begin{eqnarray}
E\left[\sup_{[0,t]}|V^{n+1}_s-V^n_s|^2\right] &\leq& 
C \intot ds \intrd f_s(dv) E[|V^n_s-V^{n-1}_s|^2
(|V^n_s-v|^\gamma+ |V^{n-1}_s-v|^\gamma) ]\ala
&&+ C E \left[ \left( \intot ds \intrd f_s(dv) |V^n_s-V^{n-1}_s|
(|V^n_s-v|^\gamma+ |V^{n-1}_s-v|^\gamma) ]\right)^2\right]\ala
&\leq & C \intot ds E[|V^n_s-V^{n-1}_s|^2]J_\gamma(f_s)\ala
&&+ C \left( \int_0^T ds J_\gamma(f_s) \right)\left(\intot ds 
E[|V^n_s-V^{n-1}_s|^2]J_\gamma(f_s) \right) \ala
&\leq & C(T,f) \intot ds E[|V^n_s-V^{n-1}_s|^2]J_\gamma(f_s).
\end{eqnarray}
Using Lemma \ref{ggl}, we deduce that $\sum_n 
E\left[\sup_{[0,T]}|V^{n+1}_s-V^n_s|^2\right] <\infty$, so that
there exists a c\`adl\`ag adapted process $(V_t)_{t\in[0,T]}$ such that
\begin{equation}\label{supconv}
\lim_n E\left[\sup_{[0,T]}|V_s-V^n_s|^2\right]=0 
\; \hbox{ and } \; E\left[\sup_{[0,T]}|V_s|^2\right]<\infty.
\end{equation}
To show that $(V_t)_{t\geq 0}$ satisfies
$MP((\ctA_t)_{t\in[0,T]}, \delta_{v_0}))$, we need to check that for all 
$0\leq s_1\leq...\leq s_l \leq s \leq t\leq T$, 
all $\phi_1,...,\phi_l \in C_c(\rd)$,
and all $\phi \in C^2_c$,
\begin{equation}\label{tt}
E\left[\left(\prod_{i=1}^l \phi_i(V_{s_i}) \right) 
\left(\phi(V_t)-\phi(V_s)-\int_s^t \ctA_u \phi(V_u)du \right) \right]=0.
\end{equation}
But we know from (\ref{rec}) that for all $n\geq 1$,
\begin{equation}\label{ttn}
E\left[\left(\prod_{i=1}^l \phi_i(V^n_{s_i}) \right) 
\left(\phi(V_t^{n+1})-\phi(V_s^{n+1})-\int_s^t \ctA_u \phi(V^n_u)du \right) 
\right]=0.
\end{equation}
It remains to pass to the limit in (\ref{ttn}) to obtain (\ref{tt}).
It suffices to use (\ref{supconv}), and to observe that 
the map $(v_u)_{u\in[0,T]} \mapsto \cK((v_u)_{u\in[0,T]}):=
\left(\prod_{i=1}^l \phi_i(v_{s_i}) \right)
\left(\phi(v_t)-\phi(v_s)-\int_s^t \ctA_u \phi(V_u)du \right)$ is continuous
on $\dt$ (endowed here with the uniform convergence) and bounded.
First, we have shown in Step 2-(b) that $\ctA_t \phi$ is bounded
by $C_\phi(1+m_2(f_t)+J_\gamma(f_t))$, and we easily deduce that $\cK$ 
is bounded
since $f\in \lipdeux \cap\lunjgamma$.
Next, the only difficulty concerning the continuity of $\cK$ is to check
that of $(v_u)_{u\in[0,T]} \mapsto \int_s^t ds \ctA_u \phi(v_u)du$.
Since $\ctA_u \phi$ is bounded
by $C_\phi(1+m_2(f_u)+J_\gamma(f_u))\in L^1([0,T])$, it suffices to check that
for each $u\in[0,T]$, $\ctA_u \phi$ is continuous on $\rd$. This follows from
Lemma \ref{acont}.
\end{proof}

We finally conclude the section with the 

\begin{proof}[Proof of Lemma \ref{cacv}]
The proof is actually almost contained in that of Lemma \ref{unili}. Indeed,
consider a weak solution $f\in \lipdeux \cap \lunjgamma$ to (\ref{be}),
and define, for each $t\in[0,T]$, the operator $\ctA_t$ by (\ref{ctat}).
We have checked in Step 6 the existence of a
solution to $MP((\ctA_t)_{t\in[0,T]}, \delta_{v_0})$.
Of course, the same arguments allow us to prove the existence of a solution
$(V_t)_{t\in[0,T]}$ to  $MP((\ctA_t)_{t\in[0,T]}, f_0)$, since
$f_0 \in \pdeux$. Consider now the law $g_t$ of $V_t$. Then taking expectations
in (\ref{mpmp}), one easily deduces that $(g_t)_{t\in[0,T]}$ solves 
(\ref{lbe}), so that due to Lemma \ref{unili}, $g=f$.
Hence for each $t\in [0,T]$, the law of $V_t$ is nothing but $f_t$.
Next, using Steps 3 and 5, we have shown how to build some $f$-Poisson measures
$N^k$ in such a way that for $(\tV^k_t)_{t\in [0,T]}$ the
$(V_0,f,k,N^k)$-process, $\lim_k E[\sup_{[0,T]} |V_t-\tV^k_t|^2]=0$.
Denoting by $f^k_t$ the law of $\tV^k_t$, this of course implies that
$\lim_k \sup_{[0,T]} W_2(f_t,f^k_t)=0$, which was our goal.
\end{proof}

\section{Applications}
\label{seccor} \setcounter{equation}{0}

We now prove our well-posedness results.
We start with the case of regularized velocity cross sections.

\begin{proof}[Proof of Corollary \ref{nul}]
We assume {\bf(A1-A2-A3-A4$(0)$)}.  We oberve that for any $g\in \pdeux$,
$J_0(g) \leq 1$. Hence for any pair of solutions $(f_t)$ and 
$(\tf_t)_{t\in[0,T]}$, and such that one of them has a density for all times,
(\ref{resultatnul}) follows immediately from Theorem
\ref{maintheo}.

Let us now prove the existence result.
We found no reference about such an existence result, but it is completely
standard. The case where $f_0 \in \pdeux$ has a finite entropy, that is
when $\intrd f_0(v) \log (1+f_0(v))dv<\infty$ can be treated following the line
of Villani \cite{villexi} (and is much more easy since we assume here
that $\Phi$ is bounded, while true soft potentials were treated there). 
The obtained solution has furthermore a finite entropy (and thus a density)
for all times.
Then the existence result for any $f_0 \in \pdeux$
is a straightforward consequence of (\ref{resultatnul}): 
it suffices to consider a sequence
$f_0^n \in \pdeux$ with finite entropy, tending to $f_0$
for the distance
$W_2$, and the associated weak solutions $(f^n_t)_{t\in [0,T]}$ to 
(\ref{wbe}).
Then (\ref{resultatnul}) ensures us that there exists $(f_t)_{t\in[0,T]}$ 
such that 
$\lim_{n \to \infty} \sup_{[0,T]} W_2(f^n_t,f_t)=0$. It is then not hard
to pass to the limit in (\ref{wbe}), and to deduce that $(f_t)_{t\in[0,T]}$
is indeed a weak solution to (\ref{wbe}).

Let us now extend (\ref{resultatnul}) to any pair of solutions
$(f_t)_{t\in[0,T]}$, $(\tf_t)_{t\in[0,T])}$, without assuming that 
$f_t$ (or $\tf_t$) has a density for all times: consider $f_0^n$ with a finite
entropy, such that $W_2(f_0,f_0^n)$ tends to $0$, and the associated
solution $(f^n_t)_{t\in [0,T]}$. Since $f^n_t$ has a finite entropy 
(and thus a density) for all times, we deduce that for all $n\geq 1$, all
$t\in [0,T]$,
\begin{eqnarray}
W_2(f_t,\tf_t) &\leq& W_2(f_t,f^n_t)+W_2(f^n_t,\tf_t)\leq
[W_2(f_0,f^n_0)+W_2(f^n_0,\tf_0)]e^{Kt}\ala
&\leq& 
[2W_2(f_0,f^n_0)+W_2(f_0,\tf_0)]e^{Kt} 
\end{eqnarray}
by the triangular inequality. Letting $n$ tend to infinity, we obtain
(\ref{resultatnul}).
The uniqueness result is now straightforward.
\end{proof}

We now study the case of soft potentials.

\begin{proof}[Proof of Corollary \ref{sp}.]
We consider $\gamma \in (-3,0)$, and assume 
{\bf (A1)-(A2)-(A3)-(A4)}$(\gamma)$. 

\vip

First note that 
we consider only solutions with densities here, since we work in 
$L^p$ with $p>3/(3+\gamma)>1$.

We also observe that for $\alpha \in (-3,0)$, and for
$p\in (3/(3+\alpha),\infty]$, there exists a constant $C_{\alpha,p}$ such that
for any $g \in \pdeux \cap L^p(\rd)$, any $v\in \rr^d$,
\begin{eqnarray}\label{suplq}
J_\alpha(g)&=&\sup_{v\in\rd}\intrd g(v_*) \, |v-v_*|^{\alpha} \, dv_* \ala
&\leq& \sup_{v\in\rd}
\int_{|v_*-v|<1} g(v_*) \, |v-v_*|^{\alpha} \, dv_* + \sup_{v\in\rd}
\int_{|v_*-v|\geq 1} 
g(v_*) \, dv_*\ala
&\leq& C_{\alpha,p} \, \| g \|_{L^p(\rd)} +1 ,
\end{eqnarray}
where 
$C_{\alpha,p}=\sup_{v\in\rd} [\int_{|v_*-v|\leq 1} 
|v-v_*|^{\alpha p/(p-1)}dv_*]^{(p-1)/p}= [\int_{|v_*|\leq 1} 
|v_*|^{\alpha p/(p-1)}dv_*]^{(p-1)/p} <\infty$, since by assumption,
$\alpha p/(p-1)>-3$.

\vip

{\it Step 1.} We first observe that point (i)
is an immediate consequence of Theorem \ref{maintheo} and (\ref{suplq}),
since we deal with solutions with densities.

\vip

{\it Step 2.} We now check point (ii). 
We only have to prove the existence of solutions, since uniqueness
follows from point (i).
Using some results of Villani \cite[Theorems 1 and 3]{villexi}, 
we know that for $\gamma \in (-3,0)$,
for any $f_0\in \pdeux$ satisfying
\begin{equation}\label{fee}
\intrd f_0(v) \log (1+f_0(v)) dv <\infty,
\end{equation}
there exists a weak solution 
$(f_t)_{t\geq 0}\in L^\infty([0,\infty), L^1(\rd,(1+|v|^2)\, dv))$ to 
(\ref{be}) starting from $f_0$.
Recall that if $f_0\in\pdeux \cap L^p(\rd)$ for some $p>1$,
then (\ref{fee}) holds.
Then the existence result of point (ii) follows immediately from the
following {\it a priori} estimate,
which guarantees that if $f_0 \in L^p(\rr^d)$ for some $p>3/(3+\gamma)$, then 
this bound propagates locally (in time): 
there exists $C=C(p,\gamma,\kappa_1,\kappa_2,\kappa_3)$ such that 
any weak solution to (\ref{be}) {\it a priori} satisfies
\begin{equation}\label{cqa}
\frac{d}{dt} \| f_t \|_{L^p} \leq C \, \big(1+ \| f_t \|_{L^p(\rd)}^2\big).
\end{equation} 
This will guarantee that for $0\leq 
t< T_*:=\frac{1}{C}(\pi/2 - \arctan ||f_0||_{L^p})$, we have 
\begin{equation}\label{tat}
\| f_t \|_{L^p} \leq \tan \left( \arctan \| f_0 \|_{L^p} + C \, t \right).
\end{equation}
Thus point (ii) will be proved.

\vip

To obtain (\ref{cqa}), we follow the method
of Desvillettes-Mouhot, see \cite[Proposition 3.2]{dm}. 
First, one classically may replace, in $\cA\phi$ taken in the form
(\ref{afini}), $\beta(\theta)$
by $\hat\beta(\theta)=
[\beta(\theta)+\beta(\pi-\theta)]\indiq_{\{\theta\in (0,\pi/2]\}}$,
see e.g. \cite[Introduction]{advw} or \cite[Section 2]{dm}. 
This relies on the use of (\ref{obs2}).
Next, following the line of 
\cite[proof of Proposition 3.2]{dm}, we get
\begin{equation}
\frac{d}{dt} \intrd |f_t(v)|^p dv \leq  (p-1) \intrd f_t(v_*)dv_*
\intrd dv \Phi(|v-v_*|) \int_0^{\pi/2}\hat\beta(\theta)d\theta
\int_0^{2\pi} d\varphi [f^p_t(v')-f_t^p(v)],
\end{equation}
where $v'$ is given by (\ref{dfvprime}). Using now the cancelation Lemma
of Alexandre-Desvillettes-Villani-Wennberg \cite[Lemma 1]{advw}
(with $N=3$, $f$ given by $f^p_t$, and $B(|v-v_*|,\cos \theta)\sin\theta=
\hat\beta(\theta) \Phi(|v-v_*|)$), we obtain
\begin{eqnarray}
&&\frac{d}{dt} \intrd |f_t(v)|^p dv \leq  2\pi(p-1) \intrd f_t(v_*)dv_*
\intrd f_t^p(v) dv \int_0^{\pi/2}\hat\beta(\theta)d\theta\ala 
&&\hskip4cm
\left|\cos^{-3}(\theta/2)\Phi(|v-v_*|\cos^{-1}(\theta/2)) - \Phi(|v-v_*|)
\right|.
\end{eqnarray}
Due to {\bf (A4)}$(\gamma)$, we know that
$|\Phi(x)-\Phi(y)|\leq \kappa_3(x^\gamma+y^\gamma)|x-y|/\min(x,y)$. One easily deduces
that for some constant $C=C(\kappa_3)$,
for all $x\in\rr_+$, all $\theta\in (0,\pi/2]$,
\begin{eqnarray}
\left|\cos^{-3}(\theta/2)\Phi(x\cos^{-1}(\theta/2)) - \Phi(x)\right|
&\leq& 
\left|\cos^{-3}(\theta/2)-1 \right| \Phi(x)\ala
&&+ \cos^{-3}(\theta/2)\left|\Phi(x\cos^{-1}(\theta/2)) - \Phi(x)\right|\ala
&\leq & C \theta^2 x^\gamma + C  |x \cos^{-1}(\theta/2)-x|  (x^\gamma + (x \cos^{-1}(\theta/2))^\gamma) / x \ala
&\leq & C \theta^2 x^\gamma.
\end{eqnarray}
we used here that $|\cos^{-3}(\theta/2)-1|+|\cos^{-1}(\theta/2)-1|\leq$
cst$\theta^2$ for all $\theta \in (0,\pi/2]$.

Since $\int_0^{\pi/2} \theta^2 \hat \beta (\theta) d\theta \leq \kappa_1$,
we finally get, with $C=C(\gamma,p,\kappa_1,\kappa_2,\kappa_3)$,
\begin{eqnarray}
\frac{d}{dt} \intrd |f_t(v)|^p dv &\leq& C
\intrd f_t^p(v) dv \intrd |v-v_*|^\gamma f_t(v_*)dv_*\ala 
&\leq &  C \left( \intrd f_t^p(v) dv + C_{\gamma,p}
\left[\intrd f_t^p(v) dv\right]^{1+1/p} \right)
\end{eqnarray}
the last inequality using (\ref{suplq}). This yields
\begin{eqnarray}
\frac{d}{dt} ||f_t||_{L^p}
&\leq& C \left( ||f_t||_{L^p}+ C_{\gamma,p}
||f_t||_{L^p}^{2} \right),
\end{eqnarray}
from which (\ref{cqa}) immediately follows.
\end{proof}

\section{Appendix}

We start this annex by recalling the generalized Gronwall Lemma and 
the associated
Picard Lemma.

\begin{lem}\label{ggl}
Let $T\geq 0$ and a nonnegative function $v$ on $[0,T]$ such that 
$\int_0^T v(s) ds <\infty$.

(i) Any nonnegative bounded function on $[0,T]$ satisfying $u(t) \leq a+
\int_0^t u(s)v(s)ds$ for all $t\in [0,T]$ 
also satisfies $u(t) \leq a \exp(\int_0^t v(s)ds)$ for all $t\in [0,T]$.

(ii) Consider a sequence of nonnegative functions $u_n$ on $[0,T]$, 
with $u_0$ bounded
and $u_{n+1}(t) \leq \int_0^t u_n(s)v(s)ds $ for all $n\geq 0$, all 
$t\in [0,T]$.
Then $\sum_{n\geq 0} \sup_{[0,T]} u_n(t) \leq ||u_0||_\infty  
\exp(\int_0^T v(s)ds)$.
\end{lem}

We carry on with the

\begin{proof}[Proof of Lemma \ref{raiso}] We first prove (i).
Recall that $H(\theta)=\int_\theta^\pi \beta(x)dx$,
that $G$ is its inverse function, and that $c\theta^{-1-\nu}\leq \beta(\theta)
\leq C\theta^{-1-\nu}$ for some $\nu\in(0,2)$.
Since $H(\theta) \leq a(\theta^{-\nu}-\pi^{-\nu})$ (with $a=C/\nu$),
we deduce that $G(z) \leq (z/a + \pi^{-\nu})^{-1/\nu}$.
Now for $0\leq z\leq w$
\begin{eqnarray}
0\leq G(z)-G(w) &=& - \int_z^w G'(u)du = \int_z^w \frac{du}{\beta(G(u))}
\leq \frac{1}{c}\int_z^w G(u)^{\nu+1} \ala 
&\leq& \frac{1}{c}\int_z^w (u/a + \pi^{-\nu})^{-1-1/\nu} 
\leq  A\left[(z/a+\pi^{-\nu})^{-1/\nu} -(w/a+\pi^{-\nu})^{-1/\nu}  \right]
\ala
&\leq& B \left[ (1+\e z)^{-1/\nu} -(1+ \e w)^{-1/\nu}   \right]
\end{eqnarray}
for some constants $A$, $B$, $\e>0$. We set $\alpha=1/\nu>1/2$, and first
treat the

\vip

{\it Case $\alpha \in (1/2,1]$}. Let thus $x\geq y>0$, and $z\in(0,\infty)$.
Using the inequality 
$|u^\alpha-v^\alpha|\leq $ cst $|u-v|/(u^{1-\alpha}+v^{1-\alpha})$,
we obtain, the value of $B$ changing from line to line,
\begin{eqnarray}
|G(z/x)-G(z/y)|^2&\leq& B | (1+\e z/x)^{-1} - (1+\e z/y)^{-1}|^2 ((1+\e z/x)^{\alpha-1}
+ (1+\e z/y)^{\alpha-1} )^{-2}\ala
&\leq & B \left| \frac{x}{x+\e z} - \frac{y}{y+\e z}\right|^2 
\left(\frac{x+\e z}{x} \right)^{2-2\alpha} \ala
&\leq & B  (x-y)^2 z^2 x^{2\alpha-2}
(x+\e z)^{-2\alpha}(y+\e z)^{-2} \ala
&\leq & B (x-y)^2 x^{2\alpha-2} (x+\e z)^{-2\alpha}.
\end{eqnarray}
Integrating this inequality, we get
\begin{eqnarray}
\int_0^\infty dz |G(z/x)-G(z/y)|^2&\leq& B (x-y)^2 x^{2\alpha-2} x^{1-2\alpha}
\leq B \frac{(x-y)^2}{x} \leq B\frac{(x-y)^2}{x+y},
\end{eqnarray}
since $x\geq y$ by assumption.

\vip

{\it Case $\alpha \in [1,\infty)$}. Let $x\geq y > 0$. Then
\begin{eqnarray}
|G(z/x)-G(z/y)|^2 &\leq& B | (1+\e z/x)^{-\alpha} - (1+\e z/y)^{-\alpha}|^2 \ala
&\leq& B | (1+\e z/x)^{-1 } - (1+\e z/y)^{-1}|^2, 
\end{eqnarray}
and we may use the same computation as previously with $\alpha=1$.

\vip

We leave the proof (ii) to the reader, and finally prove (iii).
We thus assume that $\Phi(x)=x^\gamma$ for some $\gamma \in (-3,0)$,
and show that (\ref{a4}) holds with $\Psi(x)=$cst$.x^\gamma$.
First, it is of course immediate that $(x-y)^2[\Phi(x)+\Phi(y)]
\leq (x-y)^2 [x^\gamma+y^\gamma]$. Next, 
\begin{eqnarray}
\min(x,y)|x-y||\Phi(x)-\Phi(y)|
&\leq& |\gamma| \min(x,y) (x-y)^2 \min(x,y)^{\gamma-1} \ala
&\leq& |\gamma| (x-y)^2 \min(x,y)^\gamma 
\leq  |\gamma| (x-y)^2 (x^\gamma+y^\gamma).  
\end{eqnarray}
Finally, 
\begin{eqnarray}
\min(x^2,y^2)\frac{|\Phi(x)-\Phi(y)|^2}{\Phi(x)+\Phi(y)}
&\leq& |\gamma|^2 \min(x^2,y^2) (x-y)^2 \frac{\min(x,y)^{2\gamma-2}}
{x^\gamma+y^\gamma} \ala
&\leq& |\gamma|^2 (x-y)^2 \min(x,y)^\gamma 
\leq  |\gamma|^2 (x-y)^2 (x^\gamma+y^\gamma).  
\end{eqnarray}
As a conclusion, (\ref{a4}) holds with $\Psi(x):= (1+|\gamma| + \gamma^2) 
x^\gamma$.
\end{proof}

Next, we give the 

\begin{proof}[Proof of (\ref{majfond})]
Let thus $\phi\in C^2_\infty$, denote by $\phi''$ its Hessian matrix,
and set 
$\Delta=\Delta(v,v_*,\theta,\varphi)=\phi(v')+\phi(v'_*)-\phi(v)-\phi(v_*)$,
where we used the shortened notation (\ref{dfvprime}).
Recalling that $v'=v+a$ while $v'_*=v_*-a$, a Taylor expansion yields that
 for some $w_1,w_2 \in \rd$,
$\Delta= a.[\nabla \phi(v)-\nabla\phi(v_*)]+ \frac{1}{2} a. [\phi''(w_1)+
\phi''(w_2)]a$. Recall now that 
$\int_0^{2\pi} a d\varphi = - \frac{1-\cos\theta}{2} (v-v_*)$,
and that $|a|^2= \frac{1-\cos\theta}{2} |v-v_*|^2$. Hence
\begin{eqnarray}
\left| \int_0^{2\pi} \Delta d\varphi\right| &\leq& 
\frac{1-\cos\theta}{2} |v-v_*|.|\nabla\phi(v)-\nabla\phi(v_*)| 
+ 2\pi ||\phi''||_\infty  \frac{1-\cos\theta}{2} |v-v_*|^2 \ala
&\leq & C(1-\cos\theta) ||\phi''||_\infty |v-v_*|^2,
\end{eqnarray}
which yields the desired result, since $1-\cos\theta\leq \theta^2$.
\end{proof}

Next, we treat the

\begin{proof}[Proof of Lemma \ref{rewriteA}] First, the second equality in 
(\ref{agood}) is obvious, since $c(v,v_*,z,.)$ is $2\pi$-peridodic. Next, 
we consider $\phi\in C^2_\infty$. We have already seen that $\cA \phi$ 
is well-defined
(see (\ref{majfond})), and $\ctA \phi$ is also well-defined, since, setting
$c=c(v,v_*,z,\varphi)$ for simplicity,
$|\phi(v+c)-\phi(v)-c.\nabla\phi(v)| \leq |c|^2 ||\phi''||_\infty/2$.
Using the substitution $\theta=G(z/\Phi(|v-v_*|))$, 
which yields $dz= -\Phi(|v-v_*|)\beta(\theta)d\theta$, 
we get,
\begin{eqnarray}\label{rappel1}
\int_0^\infty dz \int_0^{2\pi} d\varphi |c(v,v_*,z,\varphi)|^2
= \Phi(|v-v_*|) \int_0^{\pi} \beta(\theta)d\theta 
\int_0^{2\pi} d\varphi |a(v,v_*,\theta,\varphi)|^2 \ala
= \Phi(|v-v_*|) \int_0^{\pi} \beta(\theta)d\theta 
\int_0^{2\pi} d\varphi |v-v_*|^2 \frac{1-\cos\theta}{2}
\leq  C |v-v_*|^2 \Phi(|v-v_*|)
\end{eqnarray}
thanks to {\bf (A2)}, where $C$ depends only on $\kappa_1$. Thus $\ctA \phi$ is
well-defined for $\phi \in C^2_\infty$, and if $\phi\in C^2_b$,
\begin{equation}\label{rappel2}
|\ctA \phi(v,v_*)| \leq  C ||\phi''||_\infty |v-v_*|^2 \Phi(|v-v_*|)
+ C ||\phi'||_\infty |v-v_*| \Phi(|v-v_*|).
\end{equation}
Next, we consider again $\phi\in C^2_\infty$. Using the substitution 
$\theta=G(z/\Phi(|v-v_*|))$, we observe that (using the shortened notation
(\ref{dfvprime}))
\begin{eqnarray}\label{rappel3}
&&\ctA \phi(v,v_*)= \Phi(|v-v_*|) \left[\int_0^\pi \beta(\theta)d\theta \intzdp
d\varphi
[\phi(v')-\phi(v)-a.\nabla\phi(v)] - \kappa_0 (v-v_*).\nabla\phi(v)\right].
\end{eqnarray}
Using now (\ref{obs1}) and that $\int_0^{\pi} \beta(\theta)d\theta 
\int_0^{2\pi}d\varphi a = - \kappa_0 (v-v_*)$, we obtain
\begin{eqnarray}
&&\frac{\ctA \phi(v,v_*)+ \ctA\phi(v_*,v)}{2}
= \frac{\Phi(|v-v_*|)}{2} \int_0^\pi \beta(\theta)d\theta \intzdp d\varphi 
[\phi(v') 
+ \phi(v'_*) - \phi(v) - \phi(v_*) ],
\end{eqnarray}
which was our goal.
\end{proof}

\begin{proof}[Proof of Lemma \ref{fundest}]
First (\ref{esti1}) has already been proved, see (\ref{rappel1}).
Using again the substitution 
$\theta=G(z/\Phi(|v-v_*|))$, we obtain (\ref{esti3}):
\begin{eqnarray}
\int_0^\infty dz \left| \int_0^{2\pi} d\varphi c(v,v_*,z,\varphi) \right|
= \Phi(|v-v_*|) \int_0^\pi \beta(\theta)d\theta \left| \int_0^{2\pi} d\varphi a(v,v_*,\theta,\varphi) \right|\ala
= \pi |v-v_*| \Phi(|v-v_*|) \int_0^{\pi} \beta(\theta)d\theta 
\left[1-\cos \theta\right] \leq C |v-v_*|\Phi(|v-v_*|).
\end{eqnarray}
Next, we observe that
\begin{eqnarray}
\Delta&:=& |c(v,v_*,z,\varphi)
-c(\tv,\tv_*,z,\varphi+\varphi_0(v-v_*,\tv-\tv_*)|^2 \ala
&\leq & 4 \left|\frac{1-\cos G(z/\Phi(|v-v_*|))}{2} [(v-v_*)-(\tv-\tv_*)] ) 
\right|^2 \ala
&&+ 4 \left|\frac{\cos G(z/\Phi(|\tv-\tv_*|)) -\cos G(z/\Phi(|v-v_*|))}{2} 
[\tv-\tv_*] ) \right|^2 \ala
&&+ 4 \left|\frac{\sin G(z/\Phi(|v-v_*|))}{2} [\Gamma(v-v_*,\varphi)-
\Gamma(\tv-\tv_*,\varphi+\varphi_0(v-v_*,\tv-\tv_*))] ) 
\right|^2 \ala
&& + 4 \left|\frac{\sin G(z/\Phi(|\tv-\tv_*|)) -\sin G(z/\Phi(|v-v_*|))}{2}
 \Gamma(\tv-\tv_*,\varphi+\varphi_0(v-v_*,\tv-\tv_*) ) \right|^2.
\end{eqnarray}
Using now Lemma \ref{tanana}, that $|\Gamma(X,\varphi)|=|X|$, and
easy estimates about cosinus and sinus functions, we deduce that
\begin{eqnarray}
\Delta &\leq & C |(v-v_*)-(\tv-\tv_*)|^2 G^2(z/\Phi(|v-v_*|))\ala
&&+ C |\tv - \tv_*|^2  \left| G(z/\Phi(|v-v_*|)) - G(z/\Phi(|\tv-\tv_*|))  
\right|^2.
\end{eqnarray}
On the one hand, the substitution $\theta=G(z/\Phi(|v-v_*|))$ yields that 
$\int_0^\infty dz \int_0^{2\pi} d\varphi G^2(z/\Phi(|v-v_*|)) = 
2 \pi \Phi(|v-v_*|) \int_0^\pi \theta^2\beta(\theta)d\theta=2\pi\kappa_1
\Phi(|v-v_*|)$, and on the other hand we may use {\bf (A3)}.
We thus get
\begin{eqnarray}
\int_0^\infty dz \int_0^{2\pi} d\varphi\Delta&\leq&
C |(v-v_*)-(\tv-\tv_*)|^2 \Phi(|v-v_*|) \ala
&&+C |\tv -\tv_*|^2 \frac{(\Phi(|v-v_*|) -\Phi(|\tv-\tv_*|))^2 }
{\Phi(|v-v_*|) + \Phi(|\tv-\tv_*|)}.
\end{eqnarray}
But using a symmetry argument, we easily deduce that
\begin{eqnarray}
\int_0^\infty dz \int_0^{2\pi} d\varphi\Delta
&\leq& C |(v-v_*)-(\tv-\tv_*)|^2
(\Phi(|v-v_*|)+\Phi(|\tv-\tv_*|))\ala
&&+ C \min(|v-v_*|^2,|\tv -\tv_*|^2) 
\frac{(\Phi(|v-v_*|) -\Phi(|\tv-\tv_*|))^2 }
{\Phi(|v-v_*|) + \Phi(|\tv-\tv_*|)}.
\end{eqnarray}
Then (\ref{a4}) leads us to 
\begin{eqnarray}
\int_0^\infty dz \int_0^{2\pi} d\varphi\Delta \leq
C 
|(v-v_*)-(\tv-\tv_*)|^2 [\Psi(|v-v_*|) + \Psi(|\tv-\tv_*|)]
\end{eqnarray}
which yields (\ref{esti2}).
We finally check (\ref{esti4}). 
Integrating first against $d\varphi$, we get
\begin{eqnarray}\label{jab1}
D&:=&\int_0^\infty dz \left| \int_0^{2\pi} d\varphi [c(v,v_*,z,\varphi)
- c(\tv,\tv_*,z,\varphi)] \right| \ala
& = & \pi \int_0^\infty dz \Big| (v-v_*)[1-\cos G(z/\Phi(|v-v_*|))] -
(\tv-\tv_*)[1-\cos G(z/\Phi(|\tv-\tv_*|))] \Big| \ala
&\leq & \pi 
|(v-v_*) - (\tv-\tv_*)| \int_0^\infty dz [1-\cos G(z/\Phi(|v-v_*|))] \ala
&& + \pi |\tv-\tv_*| \int_0^\infty dz \left|\cos G(z/\Phi(|v-v_*|))- 
\cos G(z/\Phi(|\tv-\tv_*|)) \right|\ala
& \leq & \pi |(v-v_*) - (\tv-\tv_*)|
\int_0^\infty dz G^2(z/\Phi(|v-v_*|)) \ala
&&+ \pi|\tv-\tv_*|\int_0^\infty dz 
\Big|G^2(z/\Phi(|v-v_*|)) -G^2(z/\Phi(|\tv-\tv_*|)) \Big|.
\end{eqnarray}
The monotonicity of $G$ ensures us that for any $x,y>0$,
\begin{eqnarray}\label{jab2}
&&\int_0^\infty dz \Big|G^2(z/x) -G^2(z/y) \Big|=
\Big|\int_0^\infty dz G^2(z/x) - \int_0^\infty dz G^2(z/y) \Big|.
\end{eqnarray}
On the other hand, $\int_0^\infty dz G^2(z/x)=x \kappa_1$ (recall {\bf (A2)}),
thanks to the substitution $\theta=G(z/x)$.
We thus obtain
\begin{eqnarray}\label{jab3}
D&\leq& \kappa_1\pi |(v-v_*) - (\tv-\tv_*)| \Phi(|v-v_*|) 
+ \kappa_1\pi  |\tv-\tv_*|. |\Phi(|v-v_*|) - \Phi(|\tv-\tv_*|)|.
\end{eqnarray}
A symmetry argument and then (\ref{a4}) thus yields
\begin{eqnarray}
D&\leq& \kappa_1\pi |(v-v_*) - (\tv-\tv_*)| (\Phi(|v-v_*|)+ \Phi(|\tv-\tv_*|))
\ala
&&+ \kappa_1\pi \min(|v-v_*|,|\tv-\tv_*|)|\Phi(|v-v_*|) - \Phi(|\tv-\tv_*|)|
\ala
&\leq & \kappa_1\pi |(v-v_*) - (\tv-\tv_*)|(\Psi(|v-v_*|)+ \Psi(|\tv-\tv_*|)),
\end{eqnarray}
from which (\ref{esti4}) follows.

\vip

We have already checked point (ii), see (\ref{rappel2}). Let us finally
prove point (iii), following the line of \cite[Lemma 4.1]{kh}.
We thus consider $\phi\in C^2_c$, with supp $\phi \subset
\{|v|\leq x\}$. Recalling (\ref{rappel3}), we see that
\begin{eqnarray}
|\ctA\phi(v,v_*)| &\leq & \kappa_0 \Phi(|v-v_*|) |v-v_*| ||\phi'||_\infty 
\indiq_{\{|v|\leq x\}}\ala
&& + \Phi(|v-v_*|) \indiq_{\{|v|\leq 2x\}} 
\int_0^\pi \beta(\theta)d\theta \int_0^{2\pi}d\varphi
|\phi(v')-\phi(v)-a.\nabla\phi(v)| \ala
&&+ \Phi(|v-v_*|) \indiq_{\{|v|\geq 2x\}} \int_0^\pi \beta(\theta)d\theta 
\int_0^{2\pi}d\varphi |\phi(v')|=:A_1+A_2+A_3.
\end{eqnarray}
Recalling that $v'=v+a$, we deduce that $|\phi(v')-\phi(v)-a.\nabla\phi(v)|
\leq |a|^2 ||\phi''||_\infty/2$, and then recalling (\ref{rappel1}), 
we deduce that
$A_2 \leq C  ||\phi''||_\infty \Phi(|v-v_*|) |v-v_*|^2 
\indiq_{\{|v|\leq 2x\}}$.
Next, we observe that $|\phi(v')| \leq ||\phi||_\infty 
\indiq_{\{|v'|\leq x\}}$.
Since $v'=v+a$, we deduce that 
$\indiq_{\{|v|\geq 2x,|v'|\leq x\}} \leq \indiq_{\{|v|\geq 2x,|a|\geq |v|/2\}}
\leq \indiq_{\{|v|\geq 2x\}} \frac{4|a|^2}{|v|^2}$.
Thus, using (\ref{rappel1}) again,
\begin{eqnarray}
A_3 &\leq&  \Phi(|v-v_*|) \indiq_{\{|v|\geq 2x\}} \frac{4||\phi||_\infty}{|v|^2}
\int_0^\pi \beta(\theta)d\theta \int_0^{2\pi}d\varphi |a|^2 \ala
&\leq&
C \Phi(|v-v_*|) \indiq_{\{|v|\geq 2x\}} \frac{||\phi||_\infty |v-v_*|^2}
{|v|^2}.
\end{eqnarray}
As a conclusion, (\ref{atfini2}) holds, which concludes the proof.
\end{proof}

Let us now give the

\begin{proof}[Proof of Lemma \ref{moinsfundest}]
First, similarly to (\ref{rappel1}), we get
\begin{eqnarray}
\int_k^\infty dz \int_0^{2\pi} d\varphi |c(v,v_*,z,\varphi)|^2
= \Phi(|v-v_*|) \int_0^{G[k/\Phi(|v-v_*|)]} \beta(\theta)d\theta 
\int_0^{2\pi} d\varphi |a(v,v_*,\theta,\varphi)|^2 \ala
= 2\pi\Phi(|v-v_*|)  |v-v_*|^2\int_0^{G[k/\Phi(|v-v_*|)]} \beta(\theta)
 \frac{1-\cos\theta}{2}d\theta  \ala
\leq  C |v-v_*|^2 \Phi(|v-v_*|) \e_0^k(|v-v_*|)
\end{eqnarray}
by definition of $\e_0^k$, and since $(1-\cos \theta)\leq \theta^2$. 
This is nothing but (\ref{mesti1}).
Next, one easily
gets $|xh_0^k(x)-yh_0^k(y)|\leq |x-y|(h_0^k(x)+h_0^k(y))+\min(x,y)
|h_0^k(x)-h_0^k(y)|$.
On the one hand, the definition of $h_0^k$ and the substitution
$\theta=G(z/\Phi(x))$ yields $h_0^k(x)\leq \pi \Phi(x) 
\int_0^\pi (1-\cos\theta)\beta(\theta)d\theta \leq \pi \kappa_1\Phi(x)$, 
and on the other hand, $|h_0^k(x)-h_0^k(y)|\leq \pi \int_0^\infty dz|
\cos G(z/\Phi(x))
-\cos G(z/\Phi(y))|\leq C |\Phi(x)-\Phi(y)|$, recall the computations
in (\ref{jab1}-\ref{jab2}-\ref{jab3}). Hence
(\ref{a4}) yields $|xh_0^k(x)-yh_0^k(y)|\leq C |x-y|(\Phi(x)+\Phi(y)) 
+C\min(x,y)|\Phi(x)-\Phi(y)|\leq C |x-y|(\Psi(x)+\Psi(y))$, i.e. 
(\ref{mesti2}).
Next, an easy computation shows that $h_0^k(x)=\pi \Phi(x) 
\int_{G[k/\Phi(x)]}^\pi
(1-\cos\theta)\beta(\theta)d\theta$. Hence,
recalling (\ref{dfkzero}),
$|\kappa_0 x \Phi(x) - x h_0^k(x)| = x \Phi(x) \pi \int_0^{G[k/\Phi(x)]}
(1-\cos\theta)\beta(\theta)d\theta\leq x\Phi(x) \pi \e_0^k(x)$.

\vip

Finally, the fact that $\e_0^k$ is bounded by $\kappa_1$ is obvious 
from {\bf (A2)},
and for $x\geq 0$ fixed, $k/\Phi(x)$ tends to infinity, so that 
$G[k/\Phi(x)]$ tends to
$0$, and thus $\e_0^k(x)$ tends to $0$.
\end{proof}

We conclude the paper with the

\begin{proof}[Proof of Lemma \ref{acont}]
We thus assume {\bf (A1-A2-A3-A4)}$(\gamma)$ for some $\gamma \in (-3,0]$, 
and consider
$\phi \in C^2_c$, and $g\in \pdeux\cap\cJ_\gamma$. We consider a sequence 
$v_n \to v$
in $\rd$, and we have to show that $h(v_n) \to h(v)$, where $h(v):=\int_\rd
g(dv_*) \ctA \phi(v,v_*)$. Recalling (\ref{rappel3}), we write
$h(v)=h_1(v)-\kappa_0 \nabla\phi(v) . h_2(v)$, with
\begin{eqnarray}
h_1(v)&:=&\intrd g(dv_*) \Phi(|v-v_*|)\int_0^\pi \beta(\theta)d\theta \intzdp
d\varphi \Delta(v,v_*,\theta,\varphi),\ala
\Delta(v,v_*,\theta,\varphi)&:=&\phi(v+a(v,v_*,\theta,\varphi)) - \phi(v) - 
a(v,v_*,\theta,\varphi).\nabla\phi(v),\ala
h_2(v) &:=& \intrd g(dv_*) (v-v_*)\Phi(v-v_*).
\end{eqnarray}
First, due to {\bf (A4)}$(\gamma)$, one has
$|x\Phi(x)-y\Phi(y)|\leq |x-y| (\Phi(x)+\Phi(y))+\min(x,y)|\Phi(x)-\Phi(y)|\leq
C|x-y|(x^\gamma+y^\gamma)$. Thus, 
\begin{eqnarray}
&&|h_2(v_n)-h_2(v)| \leq C \intrd g(dv_*) |v_n-v| (|v_n-v_*|^\gamma + 
|v-v_*|^\gamma)
\leq C |v_n-v| J_\gamma(g) \to 0
\end{eqnarray}
as $n$ tends to infinity, since $g\in \cJ_\gamma$ by assumption.
Next, we use the map $\varphi_0$ introduced in Lemma \ref{tanana}, and 
write 
\begin{eqnarray} 
&&h_1(v_n)=\intrd g(dv_*) \Phi(|v_n-v_*|) \intzp \beta(\theta)d\theta 
\intzdp d\varphi
\Delta(v_n,v_*,\theta,\varphi+\varphi_0(v-v_*,v_n-v*)).
\end{eqnarray}
We now introduce, for $\e>0$, $h_1^\e$, which is defined as $h_1$ but 
replacing
$\Phi$ by $\Phi_\e(x):=\Phi(\max(x,\e))$. 
First, $\lim_n h_1^\e(v_n)=h_1^\e(v)$ for each $\e>0$, due to the
Lebesgue Theorem and
the following facts: 

(i) $\Phi_\e$ is continuous and bounded due to {\bf (A4)}$(\gamma)$ ;

(ii) $\lim_n \Delta(v_n,v_*,\theta,\varphi+\varphi_0(v-v_*,v_n-v*))
=\Delta(v,v_*,\theta,\varphi)$ for all $v_*,\theta,\varphi$ (because
due to Lemma \ref{tanana},
$\lim_n a(v_n,v_*,\theta,\varphi+\varphi_0(v-v_*,v_n-v*))=
a(v,v_*,\theta,\varphi)$
);

(iii) $|\Delta(v_n,v_*,\theta,\varphi+\varphi_0(v-v_*,v_n-v*))|
\leq C_\phi |v_n-v_*|^2 \theta^2 \leq C_\phi(\sup_n|v_n|^2+|v_*|^2) \theta^2$
which belongs to $L^1
(g(dv_*)\beta(\theta)d\theta d\varphi)$ due to {\bf (A2)} and since
$g\in\lipdeux$.

We thus just have to prove that $\lim_{\e\to 0} \limsup_n 
|h_1(v_n)-h_1^\e(v_n)|=0$
and $\lim_{\e\to 0} |h_1(v)-h_1^\e(v)|=0$. Using
point (iii) above and then {\bf (A2)-(A4)}$(\gamma)$, 
\begin{eqnarray}
|h_1(v_n)-h_1^\e(v_n)| &\leq& C_\phi \intrd g(dv_*) \Phi(|v_n-v_*|)
\indiq_{\{|v_n-v_*|\leq \e\}} \intzp \beta(\theta)d\theta \intzdp d\varphi 
|v_n-v_*|^2 \theta^2\ala
&\leq & C_\phi \intrd g(dv_*) |v_n-v_*|^{2+\gamma}
\indiq_{\{|v_n-v_*| \leq \e\}}
\leq C_\phi J_\gamma(g) \e^2.
\end{eqnarray}
This implies that  $\limsup_n |h_1(v_n)-h_1^\e(v_n)| \leq C_\phi 
J_\gamma(g)\e^{2} 
\to 0$ as $\e \to 0$. The same computation shows that
$\lim_{\e\to 0} |h_1(v)-h_1^\e(v)|=0$, and this concludes the proof.
\end{proof}

\def\refname{References}

\end{document}